\documentclass[a4paper,12pt]{article}
\usepackage[english]{babel}
\usepackage{tikz}
\usetikzlibrary{matrix,arrows,decorations.markings}
\usepackage{amsmath,amsfonts,amssymb,amsthm,url,textcomp}
\usepackage{csquotes}
\usepackage{a4wide}
\usepackage[numbers]{natbib}
\usepackage{fancyhdr}
\usepackage{anyfontsize}
\usepackage{hyperref}
\usepackage{hhline}
\usepackage{leftidx}

\numberwithin{theorem}{section}

\newtheorem{theoremm}{Theorem}\numberwithin{theoremm}{subsection}
\newtheorem{deffinition}[theoremm]{Definition}
\newtheorem{lemmma}[theoremm]{Lemma}
\newtheorem{corrollary}[theoremm]{Corollary}
\newtheorem{nottation}[theoremm]{Notation}
\newtheorem{propposition}[theoremm]{Proposition}
\numberwithin{theoremmm}{subsubsection}

\theoremstyle{remark}

\newtheorem{remmark}[theoremm]{Remark}
\newtheorem{exxample}[theoremm]{Example}

\newcommand{\Rad}{\operatorname{Rad}}
\newcommand{\Aut}{\operatorname{Aut}}
\newcommand{\Alt}{\mathcal{A}}
\newcommand{\PSL}{\operatorname{PSL}}

\newcommand{\sh}{\operatorname{sh}}
\newcommand{\ord}{\operatorname{ord}}
\newcommand{\Sym}{\mathcal{S}}

\newcommand{\A}{\operatorname{A}}
\renewcommand{\L}{\operatorname{L}}
\renewcommand{\l}{\operatorname{l}}
\newcommand{\C}{\operatorname{C}}

\newcommand{\mao}{\operatorname{mao}}

\newcommand{\Soc}{\operatorname{Soc}}
\newcommand{\Inn}{\operatorname{Inn}}
\newcommand{\id}{\operatorname{id}}

\newcommand{\fix}{\operatorname{fix}}
\newcommand{\Out}{\operatorname{Out}}

\newcommand{\T}{\mathcal{T}}

\newcommand{\PGL}{\operatorname{PGL}}
\newcommand{\Gal}{\operatorname{Gal}}
\newcommand{\GL}{\operatorname{GL}}
\newcommand{\Frob}{\operatorname{Frob}}

\newcommand{\im}{\operatorname{im}}

\newcommand{\G}{\mathcal{G}}
\newcommand{\fin}{\mathrm{fin}}

\newcommand{\rel}{\mathrm{rel}}

\newcommand{\Mod}[1]{\ (\textup{mod}\ #1)}
\newcommand{\cha}{\operatorname{char}}
\renewcommand{\k}{\operatorname{k}}
\newcommand{\dl}{\operatorname{dl}}
\renewcommand{\P}{\operatorname{P}}
\newcommand{\Cc}{\mathcal{C}}

\newcommand{\F}{\operatorname{F}}

\newcommand{\maxsqrt}{\operatorname{maxsqrt}}

\newcommand{\degsum}{\operatorname{degsum}}
\newcommand{\Irr}{\operatorname{Irr}}
\newcommand{\mindeg}{\operatorname{mindeg}}
\renewcommand{\F}{\mathbb{F}}
\renewcommand{\O}{\mathcal{O}}
\newcommand{\Ll}{\mathcal{L}}
\newcommand{\f}{\operatorname{f}}
\renewcommand{\t}{\operatorname{t}}
\newcommand{\Oo}{\operatorname{O}}
\newcommand{\Func}{\operatorname{F}}

\begin{document}

\title{Finite groups with an automorphism inverting, squaring or cubing a non-negligible fraction of elements}

\author{Alexander Bors\thanks{University of Salzburg, Mathematics Department, Hellbrunner Stra{\ss}e 34, 5020 Salzburg, Austria. \newline E-mail: \href{mailto:alexander.bors@sbg.ac.at}{alexander.bors@sbg.ac.at} \newline The author is supported by the Austrian Science Fund (FWF):
Project F5504-N26, which is a part of the Special Research Program \enquote{Quasi-Monte Carlo Methods: Theory and Applications}. \newline 2010 \emph{Mathematics Subject Classification}: Primary: 20D25, 20D45, 20D60. Secondary: 12E20, 20C15, 20D05. \newline \emph{Key words and phrases:} Finite groups, Automorphisms, Powers of group elements, Almost-solvability, Almost-abelianity.}}

\date{\today}

\maketitle

\abstract{There are various results in the literature which are part of the general philosophy that a finite group for which a certain parameter (for example, the number of conjugacy classes or the maximum number of elements inverted, squared or cubed by a single automorphism) is large enough must be close to being abelian. In this paper, we show the following: Fix a real number $\rho$ with $0<\rho\leq 1$. Then a finite group $G$ with an automorphism inverting or squaring at least $\rho|G|$ of the elements in $G$ is \enquote{almost abelian} in the sense that both the index and the derived length of the solvable radical of $G$ are bounded. Furthermore, if $G$ has an automorphism cubing at least $\rho|G|$ of the elements in $G$, then $G$ is \enquote{almost solvable} in the sense that the index of the solvable radical of $G$ is bounded.}

\section{Introduction}\label{sec1}

\subsection{Background and main results}\label{subsec1P1}

In the literature, there exist various results on \enquote{quantitative} conditions on finite groups $G$ that imply abelianity or a weaker property, such as nilpotency or solvability. We mention the following examples (and remark that all rational constants appearing in these results are optimal):

\begin{enumerate}
\item For $e\in\mathbb{Z}$, denote by $\L_e(G)$ the maximum number of elements of $G$ that are mapped to their $e$-th power by a single automorphism of $G$. Then either of the following implies that $G$ is abelian: $\L_{-1}(G)>\frac{3}{4}|G|$ (this is already mentioned as \enquote{known} by Miller in 1929 (see \cite[first paragraph]{Mil29a}); for a short, elementary proof, see \cite{GPa}), $\L_2(G)>\frac{1}{2}|G|$ \cite[Theorem 3.5]{Lie73a}, $\L_3(G)>\frac{3}{4}|G|$ \cite[Theorem 4.1]{Mac75a}. On the other hand, for all $e\in\mathbb{Z}\setminus\{-1,0,2,3\}$, there exists a finite nonabelian group $G_e$ such that the map $G_e\rightarrow G_e,g\mapsto g^e$, is an automorphism of $G_e$ \cite{Mil29b}. Furthermore, either of the following implies that $G$ is solvable: $\L_{-1}(G)>\frac{4}{15}|G|$ \cite[Corollary 3.2]{Pot88a}, $\L_2(G)>\frac{7}{60}|G|$ \cite[Theorem C]{Heg03a}, $\L_3(G)>\frac{4}{15}|G|$ \cite[Theorem 4.1]{Heg09a}. Finally, it is known that for solvable $G$ and fixed $\rho\in\left(0,1\right]$, a condition of the form $\L_{-1}(G)\geq\rho|G|$ implies that the derived length of $G$ is bounded from above \cite[Theorems 1.1 and 2.6]{Heg05a}.
\item Denote by $\k(G)$ the number of conjugacy classes of $G$. If $\k(G)>\frac{5}{8}|G|$, then $G$ is abelian \cite{Gus73a}. If $\k(G)>\frac{1}{2}|G|$, then $G$ is nilpotent \cite[Th{\'e}or{\`e}me 7]{Les87a} (see also \cite[Corollary 3.2]{Les95a}). If $\k(G)>\frac{1}{12}|G|$, then $G$ is solvable \cite{Les88a,Les89a} (see also the stronger result \cite[Theorem 11]{GR06a}). More generally, if $\k(G)\geq\rho|G|$ for some fixed $\rho\in\left(0,1\right]$, then both the index of the Fitting subgroup of $G$ and the derived length of the solvable radical of $G$ are bounded in terms of $\rho$ \cite[Theorem 10(ii) and Theorem 12 in combination with Lemma 2(iii)]{GR06a}.
\item Denote by $\mao(G)$ the maximum automorphism order of $G$. If $G$ is nontrivial, then $\mao(G)\leq|G|-1$, and the bound is attained if and only if $G$ is elementary abelian \cite[Theorem 2]{Hor74a}. Moreover, if $\mao(G)>\frac{1}{2}|G|$, then $G$ is abelian \cite[Theorem1.1.1(1)]{Bor15c}, if $\mao(G)>\frac{1}{10}|G|$, then $G$ is solvable \cite[Theorem 1.1.1(2)]{Bor15c}, and if $\mao(G)\geq\rho|G|$ for any fixed $\rho\in\left(0,1\right]$, then the index of the solvable radical of $G$ is bounded \cite[Theorem 1.1.1(3)]{Bor15c}.
\end{enumerate}

The uniting \enquote{philosophy} behind all the results above is that a finite group $G$ for which the parameter in question is large enough, i.e., larger than $\rho|G|$ for a fixed (large enough) $\rho\in\left(0,1\right]$, must be abelian or at least \enquote{not too far from being abelian}. Whereas, as mentioned above, results on consequences of such conditions for general, arbitrarily small $\rho\in\left(0,1\right]$ exist for the functions $\k$ and $\mao$, there are, to the author's knowledge, no such results for the functions $\L_e$ except for Hegarty's bound on the derived length for solvable $G$ and $e=-1$ mentioned above.

The purpose of this paper is to study finite groups in which $\L_e(G)\geq\rho|G|$ for a fixed, but arbitrary $\rho\in\left(0,1\right]$ and $e\in\{-1,2,3\}$. For $e=-1,2$, we will see that both the index and the derived length of the solvable radical of $G$ can be bounded. For $e=3$, we show that at least the index of the solvable radical of $G$ is bounded; whether or not such a condition is in general strong enough to impose an upper bound on the derived length of the solvable radical of $G$ remains open.

For a more concise formulation of our main results and in order to point out the connection to the \enquote{philosophy} mentioned above, we first introduce some terminology and notation. We denote the solvable radical of a finite group $G$ by $\Rad(G)$ and the derived length of a solvable group $H$ by $\dl(H)$.

\begin{deffinition}\label{almostDef}
Let $G$ be a finite group, and $I,L\in\left[0,\infty\right)$.

\begin{enumerate}
\item $G$ is called \textbf{$I$-almost-solvable} if and only if $[G:\Rad(G)]\leq I$.

\item $G$ is called \textbf{$(I,L)$-almost-abelian} if and only if $G$ is $I$-almost-solvable and $\dl(\Rad(G))\leq L$.
\end{enumerate}
\end{deffinition}

We note that the properties from Definition \ref{almostDef} can be viewed as a sort of \enquote{quantitative generalization} of solvability and abelianity respectively: They get weaker under increasing the parameter(s) involved, $1$-solvability is the same as solvability, and $(1,1)$-abelianity is the same as abelianity.

\begin{nottation}\label{constantsNot}
We define the following constants:

\begin{enumerate}
\item $E_0:=0.705(1+\log_{20160}(12))=0.8817\ldots$.

\item $E_1:=(\frac{E_0+\log_{20160}(12)+\frac{1}{3}\log_{60}(24)}{1+\log_{20160}(12)+\frac{1}{3}\log_{60}(24)}-1)^{-1}=-12.7650\ldots$.
\end{enumerate}
\end{nottation}

We now give our main results:

\begin{theoremm}\label{mainTheo}
Let $G$ be a finite group.

\begin{enumerate}
\item For all $\rho\in\left(0,1\right]$: If $\L_{-1}(G)\geq\rho|G|$, then $G$ is $(\rho^{E_1},\max(2,\log_{3/4}(2\rho)+3))$-almost-abelian.

\item For all $\rho\in\left(0,1\right]$: If $\L_2(G)\geq\rho|G|$, then $G$ is $(\rho^{-4},2\cdot\log_{3/4}(\rho)+1)$-almost-abelian.

\item There exists a function $g:\left(0,1\right]\rightarrow\left[0,\infty\right)$ such that for all finite groups $G$ and all $\rho\in\left(0,1\right]$, if $\L_3(G)\geq\rho|G|$, then $G$ is $g(\rho)$-almost-solvable.
\end{enumerate}
\end{theoremm}

We remark that the bound on the derived length of the solvable radical in Theorem \ref{mainTheo}(1) is an easy consequence of Hegarty's result \cite[Theorem 2.6]{Heg05a}; all other subresults of the theorem are essentially new.

\subsection{Overview on the paper}\label{subsec1P2}

Although our main results allow for a concise formulation and are easy to understand with some canonical knowledge in group theory, we will require a lot of tools for their proofs, \textit{inter alia} the classification of finite simple groups (CFSG), the structure theory of finite groups with trivial solvable radical and a bit of character theory. In view of this, we use Section \ref{sec2} to develop enough theory to allow for a straightforward discussion of the proofs of the main results in Section \ref{sec3}. More precisely:

\begin{itemize}
\item We begin by proving inequalities relating the number of elements squared or cubed by a finite group automorphism with the number of fixed points of that automorphism in Subsection \ref{subsec2P1}.

\item In Subsection \ref{subsec2P2}, we show how to deduce from known results that the maximum number of elements inverted by an inner automorphism of a finite group is bounded from above by the complex character degree sum of the group.

\item In Subsection \ref{subsec2P3}, we deduce some corollaries from known theorems on conjugacy class numbers and outer automorphism group orders of nonabelian finite simple groups. Together with the observations from Subsection \ref{subsec2P2}, this will allow us to show that $\L_{-1}(\Aut(S))\leq|S|^{E_0}$ for all nonabelian finite simple groups $S$.

\item Subsection \ref{subsec2P4} is dedicated to some general results allowing us to derive almost-solvability and almost-abelianity results under \enquote{suitably well-behaved} quantitative assumptions on finite groups.

\item The main result of Subsection \ref{subsec2P5} states that $\L_3(\Aut(S))/|S|\to 0$ as $|S|\to\infty$ for nonabelian finite simple groups $S$.

\item Finally, in Subsection \ref{subsec2P6}, we derive a certain coset-wise upper bound on the number of elements cubed by an automorphism of a finite group with trivial solvable radical.
\end{itemize}

As mentioned before, Section \ref{sec3} is dedicated to the proofs of the main results, and in Section \ref{sec4}, we give some concluding remarks.

\subsection{Notation and terminology}\label{subsec1P3}

We denote by $\mathbb{N}$ the set of natural numbers (including $0$) and by $\mathbb{N}^+$ the set of positive integers. The image of a set $M$ under a function $f$ is denoted by $f[M]$, and the restriction of $f$ to $M$ by $f_{\mid M}$. The image of $f$, i.e., the image of the entire domain of $f$ under $f$, is denoted by $\im(f)$. If, for $i=1,\ldots,n$, $f_i$ is a function $X_i\rightarrow Y_i$, we denote by $f_1\times\cdots\times f_n$ the function $\prod_{i=1}^n{X_i}\rightarrow\prod_{i=1}^n{Y_i}$ mapping $(x_1,\ldots,x_n)\mapsto(f_1(x_1),\ldots,f_n(x_n))$.

For functions $f,g$ mapping from an unbounded set $D$ of non-negative real numbers into $\left[0,\infty\right)$, we use the Bachmann-Landau notation \enquote{$f(x)=\Theta(g(x))$ as $x\to\infty$}, meaning that there exist positive constants $C_1$ and $C_2$ such that for all $x\in D$, $C_1f(x)\leq g(x)\leq C_2f(x)$.

For $n\in\mathbb{N}$, we denote the symmetric group and alternating group on $\{1,\ldots,n\}$ by $\Sym_n$ and $\Alt_n$ respectively. For a group $G$ and an element $g\in G$, $\tau_g:G\rightarrow G,x\mapsto gxg^{-1}$, denotes the inner automorphism of $G$ with respect to $g$, and $\ord(g)$ denotes the order of $g$. The centralizer in $G$ of an element $g\in G$ is denoted by $\C_G(g)$, and the center of $G$ by $\zeta G$. We write $H\leq G$ for \enquote{$H$ is a subgroup of $G$} and $N\cha G$ for \enquote{$N$ is a characteristic subgroup of $G$}.

Extending the notation from the beginning of this paper, we set, for an automorphism $\alpha$ of a finite group $G$ and $e\in\mathbb{Z}$, $\P_e(\alpha):=\{g\in G\mid \alpha(g)=g^e\}$, $\L_e(\alpha):=|\P_e(\alpha)|$ and $\l_e(\alpha):=\frac{1}{|G|}{\L_e(\alpha)}$. Hence $\L_e(G)=\max_{\alpha\in\Aut(G)}{\L_e(\alpha)}$, and we also set $\l_e(G):=\frac{1}{|G|}{\L_e(G)}=\max_{\alpha\in\Aut(G)}{\l_e(\alpha)}$.

For a nonzero polynomial $P(X)$ over some field $K$, we denote by $\deg(P(X))$ the degree of $P(X)$ and by $\mindeg(P(X))$ the minimum degree of a nonzero monomial of $P(X)$.

The rest of our notation is either defined at some point in the text or standard.

\subsection{Finite semisimple groups}\label{subsec1P4}

For the readers' convenience, we now briefly recall the basic theory of finite groups with trivial solvable radical. We call such groups \textit{semisimple}, in accordance with the terminology of \cite[pp. 89ff.]{Rob96a}, where one can find most of the theory mentioned below in detail.

Let $H$ be a finite semisimple group. Then the socle of $H$, $\Soc(H)$, is a finite centerless completely reducible group, i.e., it can be written as follows: $\Soc(H)=S_1^{n_1}\times\cdots\times S_r^{n_r}$, where the $S_i$ are pairwise nonisomorphic nonabelian finite simple groups, the $n_i$ are positive integers, and $r\in\mathbb{N}$ (with $r=0$ if and only if $H$ is trivial). One can show that the conjugation action of $H$ on $\Soc(H)$ is faithful, yielding an embedding $H\hookrightarrow\Aut(\Soc(H))$ whose image contains $\Inn(\Soc(H))\cong\Soc(H)$.

Conversely, if $R$ is a finite centerless completely reducible group, and $H$ is such that $\Inn(R)\leq H\leq\Aut(R)$, then $H$ is semisimple with $\Soc(H)=\Inn(R)\cong R$. Hence the finite semisimple groups are just the groups occurring in between the inner and the full automorphism group of a finite centerless completely reducible group.

Fortunately for the study of finite semisimple groups, the structure of the automorphism groups of finite centerless completely reducible groups is known: We have $\Aut(S_1^{n_1}\times\cdots\times S_r^{n_r})=\Aut(S_1^{n_1})\times\cdots\times\Aut(S_r^{n_r})$, and for a nonabelian finite simple group $S$ and $n\in\mathbb{N}^+$, we have $\Aut(S^n)=\Aut(S)\wr\Sym_n$ (permutational wreath product).

Furthermore, it follows from a result of Rose (see \cite[Lemma 1.1]{Ros75a}) that for any finite semisimple group $H$, viewing $H$ as a subgroup of $\Aut(\Soc(H))$ via the embedding mentioned above, the automorphism group of $H$ is naturally isomorphic with the normalizer of $H$ in $\Aut(\Soc(H))$. In particular, the automorphism group of a finite centerless completely reducible group is always complete.

We remark that the results of the previous paragraph imply in particular that every automorphism of a finite semisimple group $H$ extends naturally to an (inner) automorphism of $\Aut(\Soc(H))$. In particular, if $S$ is a nonabelian finite simple group, then for all $e\in\mathbb{Z}$, $\L_e(S)\leq\L_e(\Aut(S))$, a fact used at several points in the paper.

\section{Auxiliary observations}\label{sec2}

\subsection{Fixed points and elements squared or cubed by a finite group automorphism}\label{subsec2P1}

In this subsection, we establish bounds on $\L_2$- and $\L_3$-values of finite group automorphisms that involve the number of fixed points of the automorphism. This is basically due to the appearance of functions of the following type in our arguments:

\begin{nottation}\label{transNot}
Let $G$ be a group, $\alpha$ an automorphism of $G$. We denote by $\T_{\alpha}$ the function $G\rightarrow G$ mapping $g\mapsto g^{-1}\alpha(g)$.
\end{nottation}

It is well-known that the fibers of $\T_{\alpha}$ are just the right cosets of $\fix(\alpha)$, the subgroup of $G$ consisting of the fixed points of $\alpha$.

We begin with the following general bound for $\L_e(\alpha)$, which is obtained by counting conjugacy-class-wise and is a generalization of \cite[Lemma 3.3]{Lie73a}:

\begin{lemmma}\label{lELem}
Let $G$ be a finite group, $\alpha$ an automorphism of $G$ and $e\in\mathbb{Z}$. Then $\L_e(\alpha)\leq\k(G)\cdot|\fix(\alpha)|$
\end{lemmma}

\begin{proof}
It is sufficient to show that $\P_e(\alpha)$ contains at most $|\fix(\alpha)|$ many elements from each conjugacy class in $G$. Hence we fix $g\in\P_e(\alpha)$ and show that the number of conjugates $tgt^{-1}$, $t\in G$, that are also in $\P_e(\alpha)$ is at most $|\fix(\alpha)|$. This is equivalent to showing that the number of $t\in G$ such that $tgt^{-1}\in\P_e(\alpha)$ is at most $|\fix(\alpha)|\cdot|\C_G(g)|$. Now if $tgt^{-1}\in\P_e(\alpha)$, it follows that $tg^et^{-1}=\alpha(tgt^{-1})=\alpha(t)g^e\alpha(t)^{-1}$, or equivalently $t^{-1}\alpha(t)\in\C_G(g^e)$. Note that since $\alpha(g)=g^e$, $e$ and $\ord(g)$ must be coprime, and so $\C_G(g^e)=\C_G(g)$. Hence a necessary (and sufficient) condition for $tgt^{-1}\in\P_e(\alpha)$ to hold is that $t$ is in the preimage of $\C_G(g)$ under the function $\T_{\alpha}$. Since this preimage has size at most $|\C_G(g)|\cdot|\fix(\alpha)|$ by the fiber structure of $\T_{\alpha}$, we are done.
\end{proof}

Note that the upper bound on $\L_e(\alpha)$ from Lemma \ref{lELem} is good by trend if the number of fixed points of $\alpha$ is small. In the rest of this subsection, we derive other upper bounds on $\L_2$- and $\L_3$-values of finite group automorphisms, which have a tendency to be good if the number of fixed points is large. These bounds are obtained by counting $N$-coset-wise for a characteristic subgroup $N$; the following result is the basis for our argument:

\begin{propposition}\label{shiftProp}
Let $G$ be a group, $e\in\mathbb{N}^+$, $N\cha G$, and $\alpha,\beta_1,\ldots,\beta_e$ automorphisms of $G$. Set $\P_e(\alpha\mid\beta_1,\ldots,\beta_e):=\{g\in G\mid \alpha(g)=\beta_1(g)\cdots\beta_e(g)\}$. Fix $g\in\P_e(\alpha\mid\beta_1,\ldots,\beta_e)$. Then for $n\in N$, we have $ng\in\P_e(\alpha\mid\beta_1,\ldots,\beta_e)$ if and only if

\[
n\in\P_e(\alpha_{\mid N}\mid(\beta_1)_{\mid N},(\tau_{\beta_1(g)}\circ\beta_2)_{\mid N},(\tau_{\beta_1(g)}\circ\tau_{\beta_2(g)}\circ\beta_3)_{\mid N},\ldots,(\tau_{\beta_1(g)}\circ\cdots\circ\tau_{\beta_{e-1}(g)}\circ\beta_e)_{\mid N}).
\]
\end{propposition}

\begin{proof}
Under the assumptions, we have that $ng\in\P_e(\alpha\mid\beta_1,\ldots,\beta_e)$ if and only if (expressing $\alpha(ng)$ in two different ways)

\[
\beta_1(n)\beta_1(g)\beta_2(n)\beta_2(g)\cdots\beta_e(n)\beta_e(g)=\alpha(n)\beta_1(g)\beta_2(g)\cdots\beta_e(g),
\]

which is equivalent to

\begin{align*}
\alpha(n) &=\beta_1(n)\beta_1(g)\beta_2(n)\beta_2(g)\cdots\beta_{e-1}(n)\beta_{e-1}(g)\beta_e(n)\beta_{e-1}(g)^{-1}\cdots\beta_1(g)^{-1} \\ &=\beta_1(n)\cdot(\tau_{\beta_1(g)}\circ\beta_2)(n)\cdots(\tau_{\beta_1(g)}\circ\cdots\circ\tau_{\beta_{e-1}(g)}\circ\beta_e)(n).
\end{align*}
\end{proof}

Proposition \ref{shiftProp} is more general than what we need at the moment for our counting argument, but we will require the result in its full generality later in the proof of Theorem \ref{mainTheo}(2,3) (Subsections \ref{subsec3P2} and \ref{subsec3P3}). The special case of Proposition \ref{shiftProp} which we need now can be formulated more concisely using the following notation from \cite[Definition 2.1.4]{Bor15b}:

\begin{nottation}\label{shiftNot}
For a group $G$, an automorphism $\alpha$ of $G$ and $e\in\mathbb{N}$, we define a function $\sh_{\alpha}^{(e)}:G\rightarrow G$ via $\sh_{\alpha}^{(e)}(g):=g\alpha(g)\alpha^2(g)\cdots\alpha^{e-1}(g)$ (called the \textbf{$e$-th shift of $g$ under $\alpha$}).
\end{nottation}

\begin{corrollary}\label{shiftCor}
Let $G$ be a group, $e\in\mathbb{N}^+$, $N\cha G$ and $\alpha$ an automorphism of $G$. Fix $g\in\P_e(\alpha)$. Then for $n\in N$, we have $ng\in\P_e(\alpha)$ if and only if $\alpha(n)=\sh_{\tau_g}^{(e)}(n)$.
\end{corrollary}

\begin{proof}
Set $\beta_i:=\id$ for $i=1,\ldots,e$ in Proposition \ref{shiftProp}.
\end{proof}

We can use Corollary \ref{shiftCor} to prove the following upper bound on $\L_2$-values of finite group automorphisms:

\begin{lemmma}\label{lTwoLem}
Let $G$ be a finite group, $N\cha G$, let $\alpha$ be an automorphism of $G$, and denote by $\tilde{\alpha}$ the induced automorphism of $G/N$. Then $\L_2(\alpha)\leq[N:\fix(\alpha_{\mid N})]\cdot\L_2(\tilde{\alpha})$, or equivalently, $\l_2(\alpha)\leq|\fix(\alpha_{\mid N})|^{-1}\cdot\l_2(\tilde{\alpha})$.
\end{lemmma}

Observe that this implies that an automorphism of a finite group $G$ with at least $m$ fixed points can only square at most $\frac{1}{m}|G|$ elements of $G$.

\begin{proof}[Proof of Lemma \ref{lTwoLem}]
First, observe that if $g\in\P_2(\alpha)$, then $\pi(g)\in\P_2(\tilde{\alpha})$, where $\pi:G\rightarrow G/N$ is the canonical projection. Hence $\P_2(\alpha)$ can only contain elements from the $\L_2(\tilde{\alpha})$ many cosets of $N$ in $G$ corresponding to elements of $G/N$ that are squared by $\tilde{\alpha}$. Let us now bound the size of the intersection of $\P_2(\alpha)$ with an arbitrary coset $C$ of $N$. Assume that the intersection is nonempty, say containing $g$. Then upon setting $e:=2$ in Corollary \ref{shiftCor}, we find that the elements of $C=Ng$ squared by $\alpha$ are in bijective correspondence with the $n\in N$ such that $\alpha(n)=\sh_{\tau_g}^{(2)}(n)=n\tau_g(n)$, or equivalently $\T_{\alpha}(n)=\tau_g(n)$. Since $\tau_g$ is bijective, by the fiber structure of $\T_{\alpha}$, this equality can only hold for at most one $n$ from each right coset of $\fix(\alpha_{\mid N})$. Hence the total number of such $n$ is bounded from above by the number of such cosets, i.e., by $[N:\fix(\alpha_{\mid N})]$, which proves the assertion.
\end{proof}

We can also derive a similar bound for $\L_3$, but to this end, we need the following auxiliary observation:

\begin{propposition}\label{shiftTwoProp}
Let $G$ be a group, $\alpha$ an automorphism of $G$, and fix $c\in G$. Consider the map $\f_{c,\alpha}:G\rightarrow G,g\mapsto gc\alpha(g)$. Then for $g_1,g_2\in G$, we have $\f_{c,\alpha}(g_1)=\f_{c,\alpha}(g_2)$ if and only if $g_2\in\P_{-1}(\tau_{g_1}\circ\tau_c\circ\alpha)g_1$. In other words, the fibers of $\f_{c,\alpha}$ are just the subsets of $G$ of the form $\P_{-1}(\tau_g\circ\tau_c\circ\alpha)g$, $g\in G$.
\end{propposition}

\begin{proof}
Write $g_2=g_1x=yg_1$ with $x,y\in G$ (so that $x=\tau_{g_1^{-1}}(y)$). Then

\begin{align*}
\f_{c,\alpha}(g_1)=\f_{c,\alpha}(g_2) &\Leftrightarrow g_1c\alpha(g_1)=g_2c\alpha(g_2)=g_1xc\alpha(y)\alpha(g_1) \\ &\Leftrightarrow \alpha(y)=c^{-1}x^{-1}c=(\tau_{c^{-1}}\circ\tau_{g_1^{-1}})(y^{-1}) \\ &\Leftrightarrow (\tau_{g_1}\circ\tau_c\circ\alpha)(y)=y^{-1} \Leftrightarrow y\in\P_{-1}(\tau_{g_1}\circ\tau_c\circ\alpha).
\end{align*}
\end{proof}

\begin{lemmma}\label{lThreeLem}
Let $G$ be a finite group, $N\cha G$, let $\alpha$ be an automorphism of $G$ and denote by $\tilde{\alpha}$ the induced automorphism of $G/N$. Then $\L_3(\alpha)\leq[N:\fix(\alpha_{\mid N})]\cdot\L_{-1}(N)\cdot\L_3(\tilde{\alpha})$, or equivalently $\l_3(\alpha)\leq\frac{\L_{-1}(N)}{\fix(\alpha_{\mid N})}\cdot\l_3(\tilde{\alpha})$.
\end{lemmma}

\begin{proof}
Counting coset-wise just as in the proof of Lemma \ref{lTwoLem}, we see that it suffices to show that for all cosets $C$ of $N$ in $G$, we have $|C\cap\P_3(\alpha)|\leq[N:\fix(\alpha_{\mid N})]\cdot\L_{-1}(N)$. To this end, assume that $C=Ng$ with $g\in\P_3(\alpha)$. Setting $e:=3$ in Corollary \ref{shiftCor}, we find that the elements of $C$ that are also in $\P_3(\alpha)$ are in bijective correspondence with the $n\in N$ such that $\alpha(n)=\sh_{\tau_g}^{(3)}(n)=n\tau_g(n)\tau_g^2(n)$, or equivalently (using the notation from Proposition \ref{shiftTwoProp}):

\begin{equation}\label{nEq}
(\tau_{g^{-1}}\circ\T_{\alpha})(n)=\sh_{\tau_g}^{(2)}(n)=\f_{1,\tau_g}(n).
\end{equation}

Denote by $K$ the set of $n\in N$ such that Equation (\ref{nEq}) holds, and note that at least one of the $[N:\fix(\alpha_{\mid N})]$ many right cosets of $\fix(\alpha_{\mid N})$ in $N$ must contain at least $\frac{|K|}{[N:\fix(\alpha_{\mid N})]}$ many elements of $K$. Let $D$ be such a coset, and fix $n\in D\cap K$. Then for all $m\in D\cap K$, we have $\T_{\alpha}(n)=\T_{\alpha}(m)$, and hence by Equation (\ref{nEq}), $\f_{1,\tau_g}(n)=\f_{1,\tau_g}(m)$. It now follows by Proposition \ref{shiftTwoProp} that $\L_{-1}(N)\geq\L_{-1}((\tau_n\circ\tau_g)_{\mid N})\geq\frac{|K|}{[N:\fix(\alpha_{\mid N})]}$, which concludes the proof.
\end{proof}

\subsection{Character degree sums and elements inverted by an inner automorphism}\label{subsec2P2}

In this subsection, we show how to deduce nontrivial upper bounds on the maximum number of elements that an \textit{inner} automorphism of a finite group $G$ can invert; of course, for \textit{complete} $G$, these turn into upper bounds on $\L_{-1}(G)$. All this is basically just a series of applications of known results. The key role is played by Proposition \ref{sqrtProp} below.

\begin{nottation}\label{sqrtNot}
We introduce the following notation:

\begin{enumerate}
\item For a group $G$ and an element $g\in G$, set $\sqrt{g}:=\{f\in G\mid f^2=g\}\subseteq G$.

\item For a finite group $G$, set $\maxsqrt(G):=\max_{g\in G}{|\sqrt{g}|}$, the maximum number of square roots in $G$ of an element of $G$.
\end{enumerate}
\end{nottation}

\begin{propposition}\label{sqrtProp}
Let $G$ be a finite group.

\begin{enumerate}
\item For any $g\in G$, we have $\P_{-1}(\tau_g)=\sqrt{g^{-2}}\cdot g=\{rg\mid r\in G,r^2=g^{-2}\}$.

\item The maximum number of elements of $G$ inverted by an \textbf{inner} automorphism of $G$ is equal to $\maxsqrt(G)$.

\item If $G$ is complete, then $\L_{-1}(G)=\maxsqrt(G)$.
\end{enumerate}
\end{propposition}

\begin{proof}
Clearly, (2) follows from (1), and (3) follows from (2). Hence it suffices to prove (1), which follows from \cite[Proposition 2.22, equivalence of (i) and (iii)]{OS15a}.
\end{proof}

Proposition \ref{sqrtProp}(3) allows us to establish a connection to character theory, due to the following classical result:

\begin{theoremm}\label{characterTheo}
For a finite group $G$ and an irreducible $\mathbb{C}$-character $\chi$ of $G$, denote the Frobenius-Schur indicator of $\chi$ by $\nu_2(\chi)$. Then for an element $g\in G$, we have $|\sqrt{g}|=\sum_{\chi}{\nu_2(\chi)\chi(g)}$, where $\chi$ runs through the irreducible $\mathbb{C}$-characters of $G$.
\end{theoremm}

\begin{proof}
See, for example, \cite[pp. 49ff.]{Isa76a}.
\end{proof}

\begin{nottation}\label{characterNot}
Let $G$ be a finite group.

\begin{enumerate}
\item Denote by $\Irr(G)$ the set of irreducible $\mathbb{C}$-characters of $G$.

\item Set $\degsum(G):=\sum_{\chi\in\Irr(G)}{\chi(1)}$.
\end{enumerate}
\end{nottation}

\begin{corrollary}\label{characterCor}
Let $G$ be a finite complete group. Then $\L_{-1}(G)\leq\degsum(G)\leq\sqrt{\k(G)\cdot|G|}$.
\end{corrollary}

\begin{proof}
Fix $g\in G$ such that $|\sqrt{g}|=\maxsqrt(G)$, and note that by Proposition \ref{sqrtProp} and Theorem \ref{characterTheo}, we have

\begin{align*}
\L_{-1}(G) &=\maxsqrt(G)=|\sqrt{g}|=||\sqrt{g}||=|\sum_{\chi\in\Irr(G)}{\nu_2(\chi)\chi(g)}| \\ &\leq\sum_{\chi\in\Irr(G)}{|\nu_2(\chi)|\cdot|\chi(g)|}\leq\sum_{\chi\in\Irr(G)}{1\cdot\chi(1)}=\degsum(G).
\end{align*}

The inequality $\degsum(G)\leq\sqrt{\k(G)\cdot|G|}$ is a well-known application of the Cauchy-Schwarz inequality (using that $\sum_{\chi\in\Irr(G)}{\chi(1)^2}=|G|$).
\end{proof}

\subsection{Conjugacy class numbers and outer automorphism group orders in nonabelian finite simple groups \texorpdfstring{$S$}{S} with an application to bounding \texorpdfstring{$\L_{-1}(\Aut(S))$}{L(-1)(Aut(S))}}\label{subsec2P3}

We begin by citing two classical results. The first is a consequence of bounds on conjugacy class numbers due to Fulman and Guralnick \cite{FG12a} and is mentioned and used in \cite[proof of Theorem 9]{GR06a}:

\begin{theoremm}\label{fulTheo}
Let $T$ be a finite almost simple group. Then $\k(T)\leq|T|^{0.41}$.\qed
\end{theoremm}

The second result involves nice uniform bounds for conjugacy class numbers in subgroups of finite symmetric groups and in finite simple groups of Lie type, both due to Liebeck and Pyber, see \cite[Theorems 1 and 2]{LP97a}.

\begin{theoremm}\label{lieTheo}
The following hold:

\begin{enumerate}
\item Let $n\in\mathbb{N}^+$ and $H\leq\Sym_n$. Then $\k(H)\leq 2^{n-1}$.

\item Let $S$ be a finite simple group of Lie type. Denote by $l$ the untwisted Lie rank and by $q$ the field parameter of $S$. Then $\k(S)\leq(6q)^l$.\qed
\end{enumerate}
\end{theoremm}

Whereas the bound from Theorem \ref{fulTheo} yields in particular an upper bound $\k(S)\leq|S|^{0.41}$ holding for all nonabelian finite simple groups $S$, we can use those from Theorem \ref{lieTheo} to deduce the following asymptotic result:

\begin{corrollary}\label{lieCor}
Let the variable $S$ range over all nonabelian finite simple groups not isomorphic with any $\A_1(q)$ for $q$ a prime power. Then $\limsup_{|S|\to\infty}{\log_{|S|}}{\k(S)}=1/4$.
\end{corrollary}

\begin{proof}
To see that said limit superior is at least $1/4$, it suffices to observe that by \cite[Theorem 1.1(1)]{FG12a}, $\k(\A_2(q))=\Theta(q^2)$ as $q\to\infty$, whereas $|\A_2(q)|=\Theta(q^8)$ as $q\to\infty$.

The reverse inequality follows easily from Theorem \ref{lieTheo} and the CFSG.
\end{proof}

Using the observations from Subsection \ref{subsec2P2}, we can also give an upper bound on $\L_{-1}$-values of automorphism groups of nonabelian finite simple groups. To this end, we need the following result on outer automorphism group orders:

\begin{theoremm}\label{outTheo}
The following hold:
\begin{enumerate}
\item $\log_{|S|}{|\Out(S)|}\to0$ as $|S|\to\infty$ for nonabelian finite simple groups $S$.

\item $\log_{|S|}(|\Out(S))|\leq\log_{20160}(12)=0.2507106\ldots$ for all nonabelian finite simple groups $S$, with equality if and only if $S=\A_2(4)=\PSL_3(4)$.
\end{enumerate}
\end{theoremm}

\begin{proof}
This follows from the CFSG, the known formulas for $|\Out(S)|$ for nonabelian finite simple groups $S$ (for example from \cite[Table 5, p. xvi]{CCNPW85a} for the Chevalley groups) and the fact that $|\Out(S)|<\log_2(|S|)$ for all nonabelian finite simple groups $S$, proved in \cite[Theorem 1]{KohlOut}.
\end{proof}

\begin{corrollary}\label{fulCor}
Let $S$ be a nonabelian finite simple group. Then $\L_{-1}(\Aut(S))\leq|S|^{E_0}=|S|^{0.8817\ldots}$.
\end{corrollary}

\begin{proof}
Since $\Aut(S)$ is complete, we have, by Corollary \ref{characterCor} and Theorem \ref{fulTheo},

\[\L_{-1}(\Aut(S))\leq\sqrt{\k(\Aut(S))\cdot|\Aut(S)|}\leq\sqrt{|\Aut(S)|^{0.41}\cdot|\Aut(S)|}=|\Aut(S)|^{0.705},\]

which in view of Theorem \ref{outTheo} implies the assertion.
\end{proof}

\subsection{Group-theoretic functions}\label{subsec2P4}

The aim of this subsection is to develop a basic, general theory for deriving almost-solvability and almost-abelianity results under assumptions of a quantitative character on finite groups. We note that part of this theory (essentially Definition \ref{groupTheoDef}, large parts of Example \ref{groupTheoEx} as well as Lemmata \ref{almostSolvLem} and \ref{dlBoundLem}) was already included in an at present unpublished manuscript combined from the author's three arXiv preprints \cite{Bor15b,Bor15c,Bor15d} and currently under review for journal publication, but as the theory is actually never explicitly elaborated on in the arXiv preprints themselves, we give a detailed account of it here, including the new result Lemma \ref{poBoundLem}.

Our approach is rather general. Consider a function $f$ from the class $\G^{\fin}$ of finite groups into the set of non-negative real numbers. In order to derive restrictions on the structure of finite groups $G$ where $f(G)$ is sufficiently large, we require that $f$ satisfy certain inequalities relating $f(G)$ to values of $f$ on subgroups and quotients of $G$. More precisely, we introduce the following concepts:

\begin{deffinition}\label{groupTheoDef}
A function $f:\G^{\fin}\rightarrow\left[0,\infty\right)$ is called \textbf{group-theoretic} if and only if $f(G_1)=f(G_2)$ whenever $G_1$ and $G_2$ are isomorphic finite groups. Henceforth, assume that $f$ is a group-theoretic function.

\begin{enumerate}
\item $f$ is called \textbf{relative} if and only if $\im(f)\subseteq\left[0,1\right]$.

\item We define $f_{\rel}$ to be the function $\G^{\fin}\rightarrow\left[0,\infty\right),G\mapsto f(G)/|G|$.

\item $f$ is called \textbf{increasing on characteristic quotients} (\textbf{CQ-increasing}) if and only if for all finite groups $G$ and all $N\cha G$, we have $f(G)\leq f(G/N)$.

\item $f$ is called \textbf{increasing on characteristic subgroups} (\textbf{CS-increasing}) if and only if for all finite groups $G$ and all $N\cha G$, we have $f(G)\leq f(N)$.

\item $f$ is called \textbf{characteristically submultiplicative} (\textbf{C-submultiplicative}) if and only if for all finite groups $G$ and all $N\cha G$, we have $f(G)\leq f(N)\cdot f(G/N)$.
\end{enumerate}
\end{deffinition}

As for point (2) in Definition \ref{groupTheoDef}, observe that $f_{\rel}$ is of course \textit{not} relative in general, but many \enquote{natural} group-theoretic functions $f$ (such as the ones from Subsection \ref{subsec1P1}) satisfy $f(G)\leq|G|$ for all $G\in\G^{\fin}$, and for such $f$, $f_{\rel}$ \textit{is} relative.

\begin{remmark}\label{groupTheoRem}
We note the following facts following immediately from the definitions of the concepts involved:

\begin{enumerate}
\item A relative and C-submultiplicative group-theoretic function is both CS-increa-\linebreak[4]sing and CQ-increasing.

\item For a group-theoretic function $F$, $F_{\rel}$ is CQ-increasing if and only if for all finite groups $G$ and all $N\cha G$, we have $F(G)\leq|N|\cdot F(G/N)$.

\item A group-theoretic function $F$ is C-submultiplicative if and only if $F_{\rel}$ is C-submultiplicative.
\end{enumerate}
\end{remmark}

Let us illustrate the concepts introduced in Definition \ref{groupTheoDef} by means of several examples, some of which will also be of relevance later.

\begin{exxample}\label{groupTheoEx}
Consider the following examples of group-theoretic functions and their properties:

\begin{enumerate}
\item All the functions $\l_e$ are CQ-increasing. To see this, fix an automorphism $\alpha$ of $G$ such that $\L_e(\alpha)=\L_e(G)$. Just as in the proof of Lemma \ref{lTwoLem}, we see that $\P_e(\alpha)$ can only contain elements from $\L_e(\tilde{\alpha})\leq\L_e(G/N)$ many cosets of $G/N$, whence $\L_e(G)=\L_e(\alpha)\leq|N|\cdot\L_e(G/N)$.

\item The function $\L_{-1}$ (and thus $\l_{-1}$ too) is C-submultiplicative, see \cite[Lemma 1.2]{Heg05a}.

\item The function $\L_2$ is \textit{not} C-submultiplicative, since $\l_2$ is not even CS-increasing: $\l_2((\mathbb{Z}/2\mathbb{Z})^2)=1/4<5/12=\l_2(\Alt_4)$, although $\Alt_4$ contains a characteristic subgroup isomorphic with $(\mathbb{Z}/2\mathbb{Z})^2$.

\item The function $\k$ (and thus $\k_{\rel}$ too) is C-submultiplicative; actually, it even satisfies the stronger property that $\k(G)\leq\k(N)\cdot\k(G/N)$ for all finite groups $G$ and all \textit{normal} subgroups $N$ of $G$, see \cite{Gal70a}.

\item It is readily verified that for all finite groups $G$ and all normal subgroups $N$ of $G$, we have $\exp(G)\mid\exp(N)\cdot\exp(G/N)$, where $\exp$ denotes the group exponent. In particular, $\exp$ is C-submultiplicative.

\item The function $\mao_{\rel}$ is relative and CQ-increasing, see \cite[Corollary 1.1.4(1)]{Bor15d}.
\end{enumerate}
\end{exxample}

Let us now proceed to derive some lemmata to make use of the concepts introduced in Definition \ref{groupTheoDef}. The main morale of the following lemma is that deriving $C$-almost-solvability for finite groups $G$ under some constant lower bound on $f(G)$, where $f$ is CQ-increasing, is the same as bounding the order of a finite semisimple group which satisfies this bound from above by $C$:

\begin{lemmma}\label{almostSolvLem}
Let $f$ be a CQ-increasing group-theoretic function. Assume that, for finite semisimple groups $H$, we have $f(H)\to 0$ as $|H|\to\infty$. More explicitly, fix a function $g:\left(0,\infty\right)\rightarrow\left(0,\infty\right)$ such that for all $\rho\in\left(0,\infty\right)$ and all finite semisimple groups $H$ such that $f(H)\geq\rho$, we have $|H|\leq g(\rho)$.

Then if $G$ is a finite group such that $f(G)\geq\rho$, then $[G:\Rad(G)]\leq g(\rho)$.
\end{lemmma}

\begin{proof}
Since $f$ is CQ-increasing, we get that $\rho\leq f(G)\leq f(G/\Rad(G))$. Now $G/\Rad(G)$ is semisimple, and so it follows by choice of $g$ that $[G:\Rad(G)]=|G/\Rad(G)|\leq g(\rho)$.
\end{proof}

The next lemma serves to bound the derived length of the solvable radical of a finite group under suitable quantitative conditions; its proof is a generalization of \cite[proof of Theorem 1.1]{Heg05a}:

\begin{lemmma}\label{dlBoundLem}
Let $f$ be a C-submultiplicative relative group-theoretic function. Assume that there exist $k\in\mathbb{N}^+$ and $\rho_0\in\left(0,1\right)$ such that for any finite solvable group $G$ of derived length at least $k$, we have $f(G)\leq\rho_0$. Then:

\begin{enumerate}
\item For any finite solvable group $G$, we have $f(G)\leq\rho_0^{\lfloor\dl(G)/k\rfloor}$.

\item For any finite group $G$ such that $f(G)\geq\rho$, we have that $\dl(\Rad(G))\leq k\cdot\log_{\rho_0}(\rho)+k-1$.
\end{enumerate}
\end{lemmma}

\begin{proof}
For (1): The assertion is trivial if $\dl(G)<k$. On the other hand, observe that a solvable group $G$ with $\dl(G)\geq k$ has a characteristic series of length $\lfloor\dl(G)/k\rfloor$ where the factors all have derived length at least $k$. The assertion now follows by induction on the length of such a series, using C-submultiplicativity of $f$.

For (2): Since $f$ is also CS-increasing, we conclude that $\rho\leq f(G)\leq f(\Rad(G))\leq\rho_0^{\lfloor\dl(\Rad(G))/k\rfloor}$. The assertion follows upon observing that $k\cdot\lfloor\dl(\Rad(G))/k\rfloor\geq\dl(\Rad(G))-(k-1)$.
\end{proof}

For the next and final result of this section (Lemma \ref{poBoundLem}), we introduce the following terminology and notation:

\begin{deffinition}\label{poBoundDef}
Consider the following definitions:

\begin{enumerate}
\item Let $e\in\mathbb{R}$ and $\Cc$ be a class of finite groups. We say that a group-theoretic function $F$ is \textbf{$e$-power-of-order-bounded on $\Cc$}, or shortly \textbf{$e$-PO-bounded on $\Cc$}, if and only if $F(G)\leq|G|^e$ for all $G\in\Cc$.

\item We define a function $\t:\mathbb{R}\rightarrow\mathbb{R}$ via $\t(e):=\frac{e+\log_{20160}(12)+\frac{1}{3}\log_{60}(24)}{1+\log_{20160}(12)+\frac{1}{3}\log_{60}(24)}$.
\end{enumerate}
\end{deffinition}

Note that for $e<1$, we have $e<\t(e)<1$.

Now assume that $F$ is a group-theoretic function that is $1$-PO-bounded on $\G^{\fin}$ (as is the case for many \enquote{natural} examples, such as the functions $\L_e$, $e\in\mathbb{Z}$, $\k$ and $\mao$). Let us say that $F$ satisfies a \textit{nontrivial PO-bound} on a class $\Cc$ of finite groups if and only if $F$ is $e$-PO-bounded on $\Cc$ for some $e<1$. The following lemma allows us to derive a nontrivial PO-bound for $F$ on the class of all finite semisimple groups from such a bound on just the class of finite nonabelian characteristically simple groups, provided that $F$ is C-submultiplicative:

\begin{lemmma}\label{poBoundLem}
Let $F$ be a group-theoretic function. Assume that $F$ is C-submultipli-\linebreak[4]cative and $1$-PO-bounded on $\G^{\fin}$. Furthermore, assume that for some $e<1$, $F$ is even $e$-PO-bounded on the class of finite nonabelian characteristically simple groups. Then $F$ is $\t(e)$-PO-bounded on the class of finite semisimple groups.
\end{lemmma}

The following classical theorem of Dixon, which bounds the orders of solvable subgroups of finite symmetric groups, will be used in the proof of Lemma \ref{poBoundLem}:

\begin{theoremm}\label{dixonTheo}(\cite[Theorem 3]{Dix67a})
Let $n\in\mathbb{N}^+$ and $G$ be a solvable subgroup of $\Sym_n$. Then $|G|\leq 24^{(n-1)/3}$.\qed
\end{theoremm}

\begin{proof}[Proof of Lemma \ref{poBoundLem}]
Let $H$ be a finite semisimple group. We need to prove that $F(H)\leq|H|^{\t(e)}$, and we do so by induction on $|H|$. The induction base, i.e., the case $|H|=1$, follows from the general $1$-PO-boundedness of $F$.

For the induction step, write $\Soc(H)=S_1^{n_1}\times\cdots\times S_r^{n_r}$, where the $S_i$ are pairwise nonisomorphic nonabelian finite simple groups. View $H$ as a subgroup of $\Aut(\Soc(H))$. Now

\begin{align*}\Aut(\Soc(H)) &= \Aut(S_1^{n_1}\times\cdots\times S_r^{n_r})=\Aut(S_1^{n_1})\times\cdots\times\Aut(S_r^{n_r}) \\ &=(\Aut(S_1)\wr\Sym_{n_1})\times\cdots\times(\Aut(S_r)\wr\Sym_{n_r}) \\ &=(\Aut(S_1)^{n_1}\rtimes\Sym_{n_1})\times\cdots\times(\Aut(S_r)^{n_r}\rtimes\Sym_{n_r}) \\ &=(\Aut(S_1)^{n_1}\times\cdots\times\Aut(S_r)^{n_r})\rtimes(\Sym_{n_1}\times\cdots\times\Sym_{n_r}),\end{align*}

where in the last semidirect product, the action of the $i$-th factor $\Sym_{n_i}$ is by permutation of coordinates on $\Aut(S_i)^{n_i}$, and by the identity on all other $\Aut(S_j)^{n_j}$.

Thus viewing $\Aut(S_1)^{n_1}\times\cdots\times\Aut(S_r)^{n_r}$ as a subgroup of $\Aut(\Soc(H))$ as well, the intersection $K:=H\cap(\Aut(S_1)^{n_1}\times\cdots\times\Aut(S_r)^{n_r})$ is well-defined as another subgroup of $\Aut(\Soc(H))$. Note that $K$ contains $\Soc(H)$ (whence $K$ is semisimple and $\Soc(K)=\Soc(H)$) and that $K$ is characteristic in $H$ (since $\Aut(S_1)^{n_1}\times\cdots\times\Aut(S_r)^{n_r}$ is characteristic in $\Aut(\Soc(H))$ and every automorphism of $H$ extends to an automorphism of $\Aut(\Soc(H))$).

Using the C-submultiplicativity of $F$, we conclude that

\begin{equation}\label{ineq1}
F(H)\leq F(K)\cdot F(H/K).
\end{equation}

In order to facilitate reading, we will now continue to further bound each of the two factors $F(K)$ and $F(H/K)$ separately, starting with $F(K)$:

\begin{align*}
F(K) &\leq F(\Soc(H))\cdot F(K/\Soc(H))\leq F(S_1^{n_1})\cdots F(S_r^{n_r})\cdot |K/\Soc(H)| \\ &\leq|S_1|^{en_1}\cdots|S_r|^{en_r}\cdot|K/\Soc(H)|^{1-\t(e)}\cdot|K/\Soc(H)|^{\t(e)} \\ &\leq|\Soc(H)|^e\cdot|(\Aut(S_1)^{n_1}\times\cdots\times\Aut(S_r)^{n_r})/\Soc(H)|^{1-\t(e)}\cdot|K/\Soc(H)|^{\t(e)} \\ &=|\Soc(H)|^e\cdot(|\Out(S_1)|^{n_1}\cdots|\Out(S_r)|^{n_r})^{1-\t(e)}\cdot|K/\Soc(H)|^{\t(e)} \\ &=|\Soc(H)|^e(|S_1^{n_1}|^{\log_{|S_1|}(|\Out(S_1)|)}\cdots|S_r|^{\log_{|S_r^{n_r}|}(|\Out(S_r)|)})^{1-\t(e)}|K/\Soc(H)|^{\t(e)} \\ &\leq|\Soc(H)|^e\cdot(|S_1^{n_1}|^{\log_{20160}(12)}\cdots|S_r^{n_r}|^{\log_{20160}(12)})^{1-\t(e)}\cdot|K/\Soc(H)|^{\t(e)} \\ &=|\Soc(H)|^{e+\log_{20160}(12)(1-\t(e))}\cdot|K/\Soc(H)|^{\t(e)}.
\end{align*}

Here, we have used, \textit{inter alia}, that $F$ is $e$-PO-bounded on the class of finite nonabelian characteristically simple groups as well as the statement of Theorem \ref{outTheo}. In summary, we have

\begin{equation}\label{ineq2}
F(K)\leq|\Soc(H)|^{e+\log_{20160}(12)(1-\t(e))}\cdot|K/\Soc(H)|^{\t(e)}.
\end{equation}

We next bound $F(H/K)$. First, observe that by one of the isomorphism theorems,

\begin{align*}
H/K &=H/(H\cap(\Aut(S_1)^{n_1}\times\cdots\times\Aut(S_r)^{n_r})) \\ &\cong(\Aut(S_1)^{n_1}\times\cdots\times\Aut(S_r)^{n_r})H/(\Aut(S_1)^{n_1}\times\cdots\times\Aut(S_r)^{n_r}) \\ &\leq\Aut(\Soc(H))/(\Aut(S_1)^{n_1}\times\cdots\times\Aut(S_r)^{n_r})\cong\Sym_{n_1}\times\cdots\times\Sym_{n_r},
\end{align*}

so that we can view $H/K$, and hence also $\Rad(H/K)$, as a subgroup of the product $\Sym_{n_1}\times\cdots\times\Sym_{n_r}$.

By Theorem \ref{dixonTheo}, this allows us to bound the order of $\Rad(H/K)$. Indeed, denoting, for $i=1,\ldots,r$, by $\pi_i:\Sym_{n_1}\times\cdots\times\Sym_{n_r}\rightarrow\Sym_{n_i}$ the $i$-th coordinate projection, we have that

\begin{align*}
|\Rad(H/K)| &\leq|\pi_1[\Rad(H/K)]\times\cdots\times\pi_r[\Rad(H/K)]| \\ &=|\pi_1[\Rad(H/K)]|\cdots|\pi_r[\Rad(H/K)]|\leq(\sqrt[3]{24})^{n_1}\cdots(\sqrt[3]{24})^{n_r}.
\end{align*}

Hence we obtain the following for $F(H/K)$ (applying the induction hypothesis to the semisimple group $(H/K)/\Rad(H/K)$):

\begin{align*}
F(H/K) &\leq F(\Rad(H/K))\cdot F((H/K)/\Rad(H/K)) \\ &\leq|\Rad(H/K)|\cdot|(H/K)/\Rad(H/K)|^{\t(e)}=|\Rad(H/K)|^{1-\t(e)}\cdot|H/K|^{\t(e)} \\ &\leq(\sqrt[3]{24})^{n_1(1-\t(e))}\cdots(\sqrt[3]{24})^{n_r(1-\t(e))}\cdot|H/K|^{\t(e)} \\ &=|S_1|^{\log_{|S_1|}(\sqrt[3]{24})n_1(1-\t(e))}\cdots|S_r|^{\log_{|S_r|}(\sqrt[3]{24})n_r(1-\t(e))}\cdot|H/K|^{\t(e)} \\ &\leq|S_1^{n_1}|^{\frac{1}{3}\log_{60}(24)(1-\t(e))}\cdots|S_r^{n_r}|^{\frac{1}{3}\log_{60}(24)(1-\t(e))}\cdot|H/K|^{\t(e)} \\ &=|\Soc(H)|^{\frac{1}{3}\log_{60}(24)(1-\t(e))}\cdot|H/K|^{\t(e)},
\end{align*}

so in summary,

\begin{equation}\label{ineq3}
F(H/K)\leq|\Soc(H)|^{\frac{1}{3}\log_{60}(24)(1-\t(e))}\cdot|H/K|^{\t(e)}.
\end{equation}

Combining Equations (\ref{ineq1}), (\ref{ineq2}) and (\ref{ineq3}), we conclude that

\begin{align*}
F(H) &\leq|\Soc(H)|^{e+\log_{20160}(12)(1-\t(e))}|K/\Soc(H)|^{\t(e)}|\Soc(H)|^{\frac{1}{3}\log_{60}(24)(1-\t(e))}|H/K|^{\t(e)} \\ &=|\Soc(H)|^{e+(1-\t(e))(\log_{20160}(12)+\frac{1}{3}\log_{60}(24))}\cdot|K/\Soc(H)|^{\t(e)}\cdot|H/K|^{\t(e)} \\ &=|\Soc(H)|^{\t(e)}\cdot|H/\Soc(H)|^{\t(e)}=|H|^{\t(e)},
\end{align*}

where the second-to-last equality is by definition of $\t$.
\end{proof}

\subsection{On \texorpdfstring{$L_3$}{L3}-values of (almost) simple groups}\label{subsec2P5}

The results of this subsection will be used in the proof of Theorem \ref{mainTheo}(3). We begin with the following:

\begin{lemmma}\label{nonShiftLem}
Let $G$ be a finite centerless nonsolvable group. Let $\alpha,\beta,\gamma$ be automorphisms of $G$. Then there exists $g\in G$ such that $\alpha(g)\not=g\beta(g)\gamma(g)$ (i.e., $g\notin\P_3(\alpha\mid\id,\beta,\gamma)$ in the notation of Proposition \ref{shiftProp}).
\end{lemmma}

\begin{proof}
Assume otherwise. Then $\T_{\alpha}=\beta\circ\sh_{\beta^{-1}\circ\gamma}^{(2)}$, or equivalently, $\beta^{-1}\circ\T_{\alpha}=\sh_{\beta^{-1}\circ\gamma}^{(2)}=:f$. Now let $g\in G$ be arbitrary, but fixed. Since the function $\beta^{-1}\circ\T_{\alpha}$ assumes the value $f(g)$ precisely on the set $\fix(\alpha)g$, and the function $\sh_{\beta^{-1}\circ\gamma}^{(2)}$ by Proposition \ref{shiftTwoProp} assumes this value precisely on the set $\P_{-1}(\tau_g\circ\beta^{-1}\circ\gamma)g$, we conclude that for all $g\in G$, $\fix(\alpha)=\P_{-1}(\tau_g\circ\beta^{-1}\circ\gamma)$. In particular, setting $g:=1$, we find that $\beta^{-1}\circ\gamma$ inverts all fixed points of $\alpha$. Since $G$ is nonsolvable, we can fix, by \cite[Theorem]{Row95a}, a nontrivial fixed point $x$ of $\alpha$, and since $G$ is centerless, there is an element $g\in G$ such that $\tau_g(x^{-1})\not=x^{-1}$. Hence we have $x^{-1}=(\tau_g\circ\beta^{-1}\circ\gamma)(x)=\tau_g(x^{-1})\not=x^{-1}$, a contradiction.
\end{proof}

We note that the statement of Lemma \ref{nonShiftLem} applies in particular to all nonabelian finite simple groups $G$.

The rest of this subsection is dedicated to the proof of the following theorem:

\begin{theoremm}\label{almostSimpleTheo}
For all large enough nonabelian finite simple groups $S$, we have $\L_3(\Aut(S))/|S|\leq|S|^{-0.053}$. In particular, $\L_3(\Aut(S))/|S|\to0$ as $|S|\to\infty$.
\end{theoremm}

This implies that $\l_3(T)\to0$ as $|T|\to\infty$ for finite almost simple groups $T$, which in view of Lemma \ref{almostSolvLem} is a weaker form of Theorem \ref{mainTheo}(3). We will see that using what we know so far, the inequality $\L_3(\Aut(S))/|S|\leq|S|^{-0.053}$ can be verified with a rather short argument for all large enough $S$ except for those from the infinite family $\A_1(q)$, $q$ a prime power, which are settled in the following theorem:

\begin{theoremm}\label{a1Theo}
Fix $\epsilon\in\left(0,1/4\right)$. Then for all large enough prime powers $q$, we have $\L_3(\Aut(\A_1(q)))\leq q^{11/4+\epsilon}$.
\end{theoremm}

For proving Theorem \ref{a1Theo}, we require the following technical lemma, which gives, for polynomials $P(X)\in\mathbb{F}_q[X]$ satisfying a lacunarity condition of a special kind, an upper bound on the number of roots of $P(X)$ in $\mathbb{F}_q$ which is better than the trivial bound, $\deg(P(X))$.

\begin{lemmma}\label{lacunaryLem}
Let $q=p^K$ be a prime power, $L\in\mathbb{Z}$ with $\frac{3}{4}K\leq L<K$, and $0<\epsilon<\frac{1}{4}$. Furthermore, let $P(X)\in\mathbb{F}_q[X]$, and assume that $P(X)=P_1(X)+P_2(X)$, where $\deg(P_1(X))\leq q^{1/2+\epsilon}$, and $\deg(P_2(X))\leq q^{L/K}+q^{1/2+\epsilon}-1<q$, but $\mindeg(P_2(X))\geq p^L=q^{L/K}$. Then there exists a nonzero polynomial $Q(X)\in\mathbb{F}_q[X]$ of degree at most $q^{3/4+\epsilon}$ such that for all $x\in\mathbb{F}_q$, $x$ is a root of $P(X)$ if and only if it is a root of $Q(X)$. In particular, $P(X)$ has at most $q^{3/4+\epsilon}$ roots in $\mathbb{F}_q$.
\end{lemmma}

\begin{proof}
Denote by $\Frob$ the Frobenius endomorphism of the ring $\mathbb{F}_q[X]$. Set $\tilde{P_i}(X):=\Frob^{K-L}(P_i(X))$ for $i=1,2$, and let $\tilde{P}(X):=\tilde{P_1}(X)+\tilde{P_2}(X)=\Frob^{K-L}(P(X))$. Observe that for all $x\in\mathbb{F}_q$, $P(x)=0$ if and only if $\tilde{P}(x)=0$. Furthermore, we have $\deg(\tilde{P_1}(X))\leq q^{1/2+\epsilon}\cdot q^{(K-L)/K}\leq q^{3/4+\epsilon}$, and since the degrees of the nonzero monomials of $P_2(X)$ are of the form $q^{L/K}+e$ with $e\in\left[0,q^{1/2+\epsilon}-1\right]$, we find that the degrees of the nonzero monomials of $\tilde{P_2}(X)$ are of the form $q+\tilde{e}$ with $\tilde{e}\in\left[0,q^{3/4+\epsilon}-q^{(K-L)/K}\right]$. Let $T(X)\in\mathbb{F}_q[X]$ be the polynomial obtained from $\tilde{P_2}(X)$ by reducing the exponent of $X$ in each monomial modulo $q-1$ (i.e., by subtracting $q-1$). Then $\deg(T(X))\leq 1+(q^{3/4+\epsilon}-q^{(K-L)/K})<q^{3/4+\epsilon}$, and since the identity $x^q=x$ holds in $\mathbb{F}_q$, we have $T(x)=\tilde{P_2}(X)$ for all $x\in\mathbb{F}_q$. Hence, setting $Q(X):=\tilde{P_1}(X)+T(X)$, we find that $Q(X)$ has the required properties; note that $Q(X)$ is nonzero since it has only $\deg(P(X))<q$ many roots in $\mathbb{F}_q$.
\end{proof}

\begin{proof}[Proof of Theorem \ref{a1Theo}]
Write $q=p^K$, where $p$ is a prime and $K\in\mathbb{N}^+$. Denote by $\Frob$ the Frobenius automorphism of the field $\F_q$. We know that $\Aut(\A_1(q))=\Aut(\PSL_2(q))=\PGL_2(q)\rtimes\Gal(\F_q/\F_p)$, and that it is a complete group. We think of the elements of $\PGL_2(q)$ as represented by $(2\times 2)$-matrices over $\F_q$ such that either the bottom right entry is $1$ or the bottom right entry is $0$ and the top right entry is $1$ (\enquote{normalized form}). For $U\in\GL_2(q)$ (not necessarily normalized), we denote by $\overline{U}$ the image of $U$ under the canonical projection $\GL_2(q)\rightarrow\PGL_2(q)$.

Fix an (inner) automorphism $\alpha$ of $\Aut(\A_1(q))$, say the conjugation by the element $\overline{\begin{pmatrix}a & b \\ c & d\end{pmatrix}}\sigma$, where $\sigma=\Frob^L$ with $L\in\{0,1,\ldots,K-1\}$ and the coefficients $a,b,c,d$ are such that the matrix is normalized. We want to show that for large enough $q$, the number of elements $\beta=\beta(e,f,g,h,\psi)=\overline{\begin{pmatrix}e & f \\ g & h\end{pmatrix}}\psi\in\Aut(\A_1(q))$, say with $\psi=\Frob^M$ and $e,f,g,h$ such that the matrix is normalized, that are cubed by $\alpha$ (call such elements \enquote{good}) is bounded from above by $q^{11/4+\epsilon}$.

Observe first that by considering the equation $\alpha(\beta(e,f,g,h,\psi))=\beta(e,f,g,h,\psi)^3$ in $\Aut(\A_1(q))$ modulo the characteristic subgroup $\PGL_2(q)$, we find that a necessary condition for goodness of $\beta$ is that $\psi^3=\psi$, or equivalently $\psi^2=\id$. This leaves at most two possibilities for $\psi$: The identity automorphism of $\F_q$, and if $2\mid K$ (i.e., if $q$ is a square) the unique element of order $2$ in $\Gal(\F_q/\F_p)$, $\Frob^{K/2}$. Henceforth, we will always assume that $\psi^2=\id$.

Easy computations reveal that

\begin{align}\label{cubeEq}
\notag\beta^3= &\overline{\left(\begin{matrix}e^2\psi(e)+ef\psi(g)+eg\psi(f)+fg\psi(h) \\ eg\psi(e)+eh\psi(g)+g^2\psi(f)+gh\psi(h)\end{matrix}\right.} \\ &\overline{\left.\begin{matrix} ef\psi(e)+f^2\psi(g)+eh\psi(f)+fh\psi(h) \\ fg\psi(e)+fh\psi(g)+gh\psi(f)+h^2\psi(h)\end{matrix}\right)}\psi
\end{align}

(note that the matrix is broken over two lines, with the first column in the first line and the second column in the second line) and

\begin{align}\label{alphaEq}
\notag\alpha(\beta)= &\overline{\left(\begin{matrix}a\sigma(e)\psi(d)+b\sigma(g)\psi(d)-a\sigma(f)\psi(c)-b\sigma(h)\psi(c) \\ c\sigma(e)\psi(d)+d\sigma(g)\psi(d)-c\sigma(f)\psi(c)-d\sigma(h)\psi(c)\end{matrix}\right.} \\ &\overline{\left.\begin{matrix} -a\sigma(a)\psi(b)-b\sigma(g)\psi(b)+a\sigma(f)\psi(a)+b\sigma(h)\psi(a) \\ -c\sigma(e)\psi(b)-d\sigma(g)\psi(b)+c\sigma(f)\psi(a)+d\sigma(h)\psi(a)\end{matrix}\right)}\psi.
\end{align}

Note that the matrices appearing under the overlines on the right-hand sides of Equations (\ref{cubeEq}) and (\ref{alphaEq}) are in general not normalized. In order to bound the number of good elements, we partition the elements $\beta(e,f,g,h,\psi)$ of $\Aut(\A_1(q))$ such that $\psi^2=\id$ into several types (the idea being to exclude some non-generic cases from the main argument):

\begin{enumerate}
\item Type: $f=0$ or $g=0$. By our concept of normalized form, the assumption implies that $h=1$, so there are at most $4q^2$ such elements in total (at most $2q^2$ for each of the two cases $f=0$ resp. $g=0$, where the factor $2$ comes from the two choices for $\psi$, and the factor $q^2$ is an upper bound for the number of choices for $(e,g)$ resp.~$(e,f$)). In particular, there are at most $4q^2$ good elements of that type.

\item Type: $f,g\not=0$ and $fg\psi(e)+fh\psi(g)+gh\psi(f)+h^2\psi(h)=0$. Thinking of $f,g,h,\psi$ as fixed, the last assumption becomes a nonzero polynomial equation in $e$ of degree $p^M\leq q^{1/2}$, so for at most $q^{1/2}$ values of $e$, the resulting element $\beta$ is of this type. Hence there are at most $2(q^2+q)q^{1/2}$ elements of that type in total, in particular at most that many good elements of that type.

\item Type: $f,g\not=0$ and $fg\psi(e)+fh\psi(g)+gh\psi(f)+h^2\psi(h)\not=0$. Note that if $\beta$ is to be a good element of that type, it follows that $-c\sigma(e)\psi(b)-d\sigma(g)\psi(b)+c\sigma(f)\psi(a)+d\sigma(h)\psi(a)\not=0$ as well. This allows us to normalize the matrices occurring on the right-hand sides of Equations (\ref{cubeEq}) and (\ref{alphaEq}) by dividing through the bottom right entry. We may then compare the top left entries of the normalized matrices to obtain the following necessary condition for goodness of $\beta$:

\begin{align}\label{ePoly}
\notag &(e^2\psi(e)+ef\psi(g)+eg\psi(f)+fg\psi(h)) \\ &\cdot(-c\sigma(e)\psi(b)-d\sigma(g)\psi(b)+c\sigma(f)\psi(a)+d\sigma(h)\psi(a)) \notag \\ &=(fg\psi(e)+fh\psi(g)+gh\psi(f)+h^2\psi(h)) \notag \\ &\cdot(a\sigma(e)\psi(d)+b\sigma(g)\psi(d)-a\sigma(f)\psi(c)-b\sigma(h)\psi(c)).
\end{align}

We will now bound the number of elements $\beta$ such that Equation (\ref{ePoly}) holds. To this end, we make a case distinction:

\begin{enumerate}
\item Case: $b,c\not=0$. View $f,g,h,\psi$ as fixed. We want to bound the number of $e\in\mathbb{F}_q$ such that Equation (\ref{ePoly}) holds. It is easy to check by the assumptions that Equation (\ref{ePoly}) is a polynomial equation in $e$ of degree precisely $p^L+p^M+2$. We distinguish two subcases:

\begin{itemize}
\item Subcase: $L\leq\frac{3}{4}K$. Then Equation (\ref{ePoly}) is a nonzero polynomial equation in $e$ of degree at most $q^{3/4}+q^{1/2}+2\leq q^{3/4+\epsilon/3}$ for $q$ large enough. Hence the number of good elements of Type 3 is bounded from above by $2(q^2+q)q^{3/4+\epsilon/3}\leq q^{11/4+\epsilon/2}$ for $q$ large enough, and so the total number of good elements is, still for large enough $q$, bounded from above by $4q^2+2(q^2+q)q^{1/2}+q^{11/4+\epsilon/2}\leq q^{11/4+\epsilon}$, as required.

\item Subcase: $L>\frac{3}{4}K$. Note that the nonzero $e$-monomials occurring in the polynomial Equation (\ref{ePoly}) have degrees among the following numbers: $p^L+p^M+2$, $p^L+p^M$, $p^L+1$, $p^L$, $p^M+2$, $p^M$, $1$ and $0$. Hence for large enough $q$, we find that Equation (\ref{ePoly}) is equivalent to a condition of the form $P(e)=0$, where $P(X)\in\mathbb{F}_q[X]$ depends on $f,g,h,\psi$ and satisfies the lacunarity assumptions of Lemma \ref{lacunaryLem} with $\epsilon$ replaced by $\epsilon/3$. Hence by an application of Lemma \ref{lacunaryLem}, we can conclude just as in the first subcase.
\end{itemize}

\item Case: $b=0$ or $c=0$. We note that by our concept of normalization, the case assumption implies that $a\not=0$ and $d=1$. Observe also that it implies that the first summand $-c\sigma(e)\psi(b)$ of the second factor on the left-hand side of Equation (\ref{ePoly}) vanishes, turning the factor into a constant $e$-polynomial distinct from $0$, so that the $e$-polynomial on the left-hand side of Equation (\ref{ePoly}) now only has degree $p^M+2$, whereas the polynomial on the right-hand side has degree $p^M+p^L$. Hence if $L>1$, and thus $p^M+p^L>p^M+2$, then Equation (\ref{ePoly}) is equivalent to a condition of the form $P(e)=0$, where $P(X)\in\mathbb{F}_q[X]$ is of degree $p^M+p^L$, and we can conclude as in Case (a) (distinguishing between the subcases $L\leq\frac{3}{4}K$ and $L>\frac{3}{4}K$).

It remains to discuss the two cases $L=0$ and $L=1$. Note that we are done once we have shown that for each choice of $f,g,h,\psi$, there are, for large enough $q$, at most $q^{3/4+\epsilon/3}$ many choices of $e$ such that Equation (\ref{ePoly}) holds. For $L=0$, Equation (\ref{ePoly}) is a polynomial equation in $e$ of degree $p^M+2\leq q^{1/2+\epsilon/2}$ for $q$ large enough, whence we are done. For $L=1$, which implies that $K\geq 2$, and assuming $p>2$, Equation (\ref{ePoly}) is polynomial in $e$ of degree $p^M+p\leq 2q^{1/2}\leq q^{1/2+\epsilon/3}$ for $q$ large enough, whence we are also done.

Let us now discuss the final case, $L=1$ and $p=2$. Then both sides of Equation (\ref{ePoly}) are polynomials in $e$ of degree $2^M+2\leq q^{1/2+\epsilon/3}$ for $q=2^K$ large enough, so all that we need to show is that for each choice of $f,g,h,\psi$, there always is at least one $e$-monomial that does not cancel when subtracting the monomials of one side of Equation (\ref{ePoly}) from the other. This is certainly true when the leading coefficients are distinct, so we assume that they are equal. Furthermore, note that if $fh\psi(g)+gh\psi(f)+h^2\psi(h)\not=0$, then the RHS of Equation (\ref{ePoly}) has an $e$-monomial of degree $2$, which the LHS does not have, and we are done. Hence assume $fh\psi(g)+gh\psi(f)+h^2\psi(h)=0$. Then if $h\not=0$, the LHS has a constant $e$-monomial, but the RHS does not have such a monomial, whence we may even assume that $h=0$. Henceforth, we assume additionally that $b=0$; the case $c=0$ works analogously. Note that under this additional assumption, by comparing the leading coefficients of the two sides of Equation (\ref{ePoly}), we obtain $cf^2a^{2^M}=fga$, which implies that $c\not=0$. Moreover, we know that $M\in\{0,\frac{K}{2}\}$, but if $M=\frac{K}{2}$, then the LHS of Equation (\ref{ePoly}), in contrast to the RHS, does not have an $e$-monomial of degree $2^M$. Therefore, we may assume $M=0$ from now on. Then the coefficient of the $e$-monomial of degree $1$ on the LHS of Equation (\ref{ePoly}) equals $(fg+gf)\cdot c\sigma(f)\psi(a)=0$, and so the $e$-monomial of degree $1$ on the RHS does not cancel, which concludes the proof.
\end{enumerate}
\end{enumerate}
\end{proof}

\begin{proof}[Proof of Theorem \ref{almostSimpleTheo}]
By Theorem \ref{a1Theo} and observing that $|\A_1(q)|=\Theta(q^3)$ as $q\to\infty$, we may assume that $S$ is not isomorphic with any $\A_1(q)$, $q$ a prime power. Then by Theorem \ref{outTheo}(1) and Corollary \ref{lieCor}, we have $|\Out(S)|\leq|S|^{0.001}$ and $\k(S)\leq|S|^{0.26}$, provided that $S$ is large enough. Fix an automorphism $\alpha$ of $\Aut(S)$ such that $\L_3(\alpha)=\L_3(\Aut(S))$. Distinguish two cases:

\begin{enumerate}
\item Case: $|\fix(\alpha_{\mid S})|\geq |S|^{0.685}$. Then we have

\begin{align*}
\frac{\L_3(\Aut(S))}{|S|} &=\frac{\L_3(\alpha)}{|S|}\leq\frac{\L_3(\Out(S))\cdot\L_{-1}(S)}{|\fix(\alpha_{\mid S})|}\leq\frac{|\Out(S)|}{|\fix(\alpha_{\mid S})|}\cdot\L_{-1}(\Aut(S)) \\ &\leq\frac{|\Out(S)|}{|\fix(\alpha_{\mid S})|}\cdot\sqrt{\k(\Aut(S))\cdot|\Aut(S)|} \\ &\leq\frac{|\Out(S)|}{|\fix(\alpha_{\mid S})|}\cdot\sqrt{\k(S)}\cdot|\Out(S)|\cdot\sqrt{|S|} \\ &\leq|S|^{0.001+0.13+0.001+0.5-0.685}
=|S|^{-0.053}.\end{align*}

\item Case: $|\fix(\alpha_{\mid S})|<|S|^{0.685}$. Then $|\fix(\alpha)|\leq|\fix(\alpha_{\mid S})|\cdot|\Out(S)|<|S|^{0.686}$. It follows that

\begin{align*}
\L_3(\Aut(S)) &\leq\k(\Aut(S))\cdot|\fix(\alpha)|\leq\k(S)\cdot|\Out(S)|\cdot|\fix(\alpha)| \\ &<|S|^{0.26+0.001+0.686}=|S|^{0.947},
\end{align*}

and so $\L_3(\Aut(S))/|S|<|S|^{-0.053}$.
\end{enumerate}
\end{proof}

\subsection{\texorpdfstring{$\L_3$}{L3}-values of finite semisimple groups and coordinate dependencies}\label{subsec2P6}

We begin by fixing the meaning of some variables. Let $H$ be a finite semisimple group with characteristically simple socle, say $\Soc(H)=S^n$ for a nonabelian finite simple group $S$. Viewing, just as in the proof of Lemma \ref{poBoundLem}, $H$ as a subgroup of $\Aut(\Soc(H))=\Aut(S)^n\rtimes\Sym_n$, we set $K:=H\cap\Aut(S)^n$. Note that $S^n\leq K$, and that elements of $H$ lie in the same coset of $K$ if and only if their images under the canonical projection $\pi:\Aut(\Soc(H))\rightarrow\Sym_n$ are equal. Fix an automorphism $\alpha$ of $H$, which is already determined by its restriction to $\Soc(H)=S^n$, so that we can write (identifying $\alpha$ with that restriction) $\alpha=(\alpha_1\times\cdots\times\alpha_n)\circ\sigma_{\alpha}$, where $\alpha_i$ is an automorphism of $S$ for $i=1,\ldots,n$, and $\sigma_{\alpha}$ is a coordinate permutation on $S^n$ identified with an element of $\Sym_n$.

Now assume that $\beta\in H$ is cubed by $\alpha$. Just like $\alpha$, write $\beta=(\beta_1\times\cdots\times\beta_n)\circ\sigma_{\beta}$. Our goal in this subsection is to establish an upper bound on the number of elements in the coset $K\beta$ that are cubed by $\alpha$ based on the cycle structures of $\sigma_{\alpha}$ and $\sigma_{\beta}$. This will be used in the proof of Theorem \ref{mainTheo}(3) in a $K$-coset-wise counting argument.

To this end, set $K_{\beta}:=\{k\in K\mid \alpha(k\beta)=(k\beta)^3\}$. An application of Corollary \ref{shiftCor} with $e:=3$ yields:

\begin{propposition}\label{coordinateProp}
Let $k=(k_1,\ldots,k_n)\in K\subseteq\Aut(S)^n$. Then $k\in K_{\beta}$ if and only if $\alpha(k)=k\beta(k)\beta^2(k)$, i.e., if and only if for all $i=1,\ldots,n$, we have

\begin{equation}\label{coordinateEq}
\alpha_i(k_{\sigma_{\alpha}^{-1}(i)})=k_i\cdot\beta_i(k_{\sigma_{\beta}^{-1}(i)})\cdot(\beta_i\circ\beta_{\sigma_{\beta^{-1}(i)}})(k_{\sigma_{\beta}^{-2}(i)}).
\end{equation}\qed
\end{propposition}

The idea now is to derive dependencies between certain coordinates of elements $k=(k_1,\ldots,k_n)\in K_{\beta}$ based on Equation (\ref{coordinateEq}). More precisely, we will work with the following concept (defining, or $i=1,\ldots,n$, $\pi_i$ as the $i$-th coordinate projection $\Aut(S)^n\rightarrow\Aut(S)$):

\begin{deffinition}\label{determinedDef}
Let $I\subseteq\{1,\ldots,n\}$, say $I=\{i_1,\ldots,i_j\}$ with $i_1<i_2<\cdots<i_j$, and let $F\subseteq K$ and $C>0$. We say that $F$ is \textbf{$C$-determined by $I$} if and only if for all $k_{i_1},\ldots,k_{i_j}\in\Aut(S)$, there are at most $C$ elements $f\in F$ such that $\pi_{i_l}(f)=k_{i_l}$ for all $l=1,\ldots,j$.
\end{deffinition}

\begin{propposition}\label{determinedProp}
Let $I\subseteq\{1,\ldots,n\}$, $C>0$, and assume that $F\subseteq K$ is $C$-determined by $I$. Then $|F|\leq\frac{C}{|S|^{n-|I|}}\cdot|K|$.
\end{propposition}

\begin{proof}
Say $I=\{i_1,\ldots,i_j\}$, $i_1<i_2<\cdots<i_j$. For proving the assertion, it is sufficient to give exact covers $(F_{\vec{k}})_{\vec{k}\in\Aut(S)^j}$ and $(K_{\vec{k}})_{\vec{k}\in\Aut(S)^j}$ of $F$ and $K$ respectively such that for all $\vec{k}\in\Aut(S)^j$, $|F_{\vec{k}}|\leq\frac{C}{|S|^{n-|I|}}|K_{\vec{k}}|$. To this end, define, for $X\in\{F,K\}$ and $\vec{k}=(k_{i_1},\ldots,k_{i_j})\in\Aut(S)^j$, $X_{\vec{k}}$ as the set of those $x\in X$ such that $\pi_{i_l}(x)=k_{i_l}$ for $l=1,\ldots,j$.

Clearly, $(X_{\vec{k}})_{\vec{k}\in\Aut(S)^j}$ is an exact cover of $X$ for $X=F,K$, and $F_{\vec{k}}\subseteq K_{\vec{k}}$ for all $\vec{k}\in\Aut(S)^j$. Hence the asserted inequality concerning the cardinalities of $F_{\vec{k}}$ and $K_{\vec{k}}$ is trivial if $K_{\vec{k}}=\emptyset$. On the other hand, if $K_{\vec{k}}\not=\emptyset$, the inequality follows by observing that $|F_{\vec{k}}|\leq C$ by assumption, whereas $|K_{\vec{k}}|\geq|S|^{n-j}$, since if $k\in K_{\vec{k}}$, we also have $kt\in K_{\vec{k}}$ for all $t\in S^n$ with $\pi_{i_l}(t)=1$ for $l=1,\ldots,j$.
\end{proof}

In order to apply Proposition \ref{determinedProp} for our problem of bounding $|K_{\beta}|$, we introduce the following notions:

\begin{deffinition}\label{opportuneDef}
Consider the following definitions:

\begin{enumerate}
\item We call an index $i\in\{1,\ldots,n\}$ \textbf{opportune} if and only if $i$ is not a common fixed point of $\sigma_{\alpha}$ and $\sigma_{\beta}$.

\item For an opportune index $i\in\{1,\ldots,n\}$, we call the set $\Oo_i:=\{i,\sigma_{\alpha}^{-1}(i),\sigma_{\beta}^{-1}(i),\linebreak[4]\sigma_{\beta}^{-2}(i)\}$ an \textbf{opportune index set}.
\end{enumerate}
\end{deffinition}

The next lemma provides the aforementioned bound on $|K_{\beta}|$ based on the cycle structures of $\sigma_{\alpha}$ and $\sigma_{\beta}$:

\begin{lemmma}\label{opportuneLem}
The following hold:

\begin{enumerate}
\item Let $M\in\mathbb{N}$. If there are at least $M$ opportune $i\in\{1,\ldots,n\}$, then there exists a family of $\lceil\frac{M}{16}\rceil$ pairwise disjoint opportune index sets.

\item If $O_1,\ldots,O_t$ are $t$ pairwise disjoint opportune index sets, then there exist $o_i\in O_i$, $i=1,\ldots,t$, such that $K_{\beta}$ is $|S|^{0.882t}$-determined by $\{1,\ldots,n\}\setminus\{o_1,\ldots,o_t\}$.

\item Let $M\in\mathbb{N}$, $M\geq 1$. If there are at least $M$ opportune $i\in\{1,\ldots,n\}$, then $\frac{|K_{\beta}|}{|K|}\leq|S|^{-0.118\lceil M/16\rceil}\leq\min(|S|^{-0.118},60^{-0.118\lceil M/16\rceil})$.
\end{enumerate}
\end{lemmma}

\begin{proof}
(3) follows from (1) and (2) via Proposition \ref{determinedProp}.

For (1): This is clear if $M\leq 16$, so assume that $M>16$. Define $\O_0$ as the set of opportune indices $i\in\{1,\ldots,n\}$. Note that $|\O_0|\geq M$ by assumption, fix any $i_0\in\O_0$, and set $\Omega_0:=\Oo_{i_0}$. Assume now that, for some $k\in\{0,\ldots,\lceil M/16\rceil-2\}$, we have already defined a decreasing chain of $k+1$ sets of opportune indices $\O_0\supseteq\O_1\supseteq\cdots\supseteq\O_k$ such that $|\O_k|\geq M-16k$. Assume further that we have defined $k+1$ pairwise disjoint opportune index sets $\Omega_0,\ldots,\Omega_k$ such that for $l=0,\ldots,k-1$, $(\Omega_l\cup\sigma_{\alpha}[\Omega_l]\cup\sigma_{\beta}[\Omega_l]\cup\sigma_{\beta}^2[\Omega_l])\cap\O_{l+1}=\emptyset$. Then set $\O_{k+1}:=\O_k\setminus(\Omega_k\cup\sigma_{\alpha}[\Omega_k]\cup\sigma_{\beta}[\Omega_k]\cup\sigma_{\beta}^2[\Omega_k])$. Noting that $|\O_{k+1}|\geq|\O_k|-16\geq M-16(k+1)>0$, fix any $i_{k+1}\in\O_{k+1}$ and set $\Omega_{k+1}:=\O_{i_{k+1}}$. It is easy to check that with these definitions, the recursive construction is continued, and we can use this to construct $\lceil M/16\rceil$ pairwise disjoint opportune index sets $\Omega_0,\Omega_1,\ldots,\Omega_{\lceil M/16\rceil-1}$, as required.

For (2): By Corollary \ref{fulCor}, it suffices to show that for any opportune index $i$, there exists $r\in \Oo_i=\{i,\sigma_{\alpha}^{-1}(i),\sigma_{\beta}^{-1}(i),\sigma_{\beta}^{-2}(i)\}$ such that upon fixing, for each $t\in \Oo_i\setminus\{r\}$, an element $k_t\in\Aut(S)$, there exists a subset $R\subseteq\Aut(S)$ of size at most $\L_{-1}(\Aut(S))$ such that for all $k\in K_{\beta}$ with $\pi_t(k)=k_t$ for all $t\in \Oo_i\setminus\{r\}$, we have $\pi_r(k)\in R$. We prove the existence of such an $R$ in a case distinction. Write $k=(k_1,\ldots,k_n)$.

\begin{enumerate}
\item Case: $\sigma_{\alpha}^{-1}(i)\notin\{i,\sigma_{\beta}^{-1}(i),\sigma_{\beta}^{-2}(i)\}$. Then by Equation (\ref{coordinateEq}), we see that the $\sigma_{\alpha}^{-1}(i)$-th coordinate of $k$ is fully determined by the $i$-th, $\sigma_{\beta}^{-1}(i)$-th and $\sigma_{\beta}^{-2}(i)$-th coordinates of $k$. In particular, we can choose $R$ of cardinality $1\leq\L_{-1}(\Aut(S))$ in this case.

\item Case: $\sigma_{\alpha}^{-1}(i)\in\{i,\sigma_{\beta}^{-1}(i),\sigma_{\beta}^{-2}(i)\}$. We distinguish further according to the length $l$ of the cycle of $i$ under $\sigma_{\beta}$.

\begin{itemize}
\item Subcase: $l\geq3$. Then it is not difficult to see that one can always isolate one of the three distinct coordinates of $k$ appearing in Equation (\ref{coordinateEq}), so that one can again choose $R$ of cardinality $1$. For example, if $\sigma_{\alpha}^{-1}(i)=i$, Equation (\ref{coordinateEq}) is equivalent to

\[k_{\sigma_{\beta}^{-2}(i)}=(\beta_{\sigma_{\beta}^{-1}(i)}^{-1}\circ\beta_i^{-1})(\beta_i(k_{\sigma_{\beta}^{-1}(i)})^{-1}k_i^{-1}\alpha_i(k_i)).\]

\item Subcase: $l=2$. Then $\sigma_{\beta}^{-2}(i)=i$, so there are only two distinct coordinates in Equation (\ref{coordinateEq}) now, $k_i$ and $k_{\sigma_{\beta}^{-1}(i)}$. If $\sigma_{\alpha}^{-1}(i)=i$, one can isolate $k_{\sigma_{\beta}^{-1}(i)}$, and if $\sigma_{\alpha}^{-1}(i)=\sigma_{\beta}^{-1}(i)=:j$, Equation (\ref{coordinateEq}) turns into

\[\alpha_i(k_j)=k_i\cdot\beta_i(k_j)\cdot(\beta_i\circ\beta_j)(k_i).\]

Hence upon fixing the value of the coordinate $k_j$, the terms $\alpha_i(k_j)=:C_1$ and $\beta_i(k_j)=:C_2$ become constants, and we see that a possible choice for $R$ is the fiber of $C_1$ under the function $\f_{C_2,\beta_i\circ\beta_j}:\Aut(S)\rightarrow\Aut(S)$ (notation as in Proposition \ref{shiftTwoProp}). By Proposition \ref{shiftTwoProp}, this fiber has cardinality bounded from above by $\L_{-1}(\Aut(S))$, as required.

\item Subcase: $l=1$. This subcase cannot occur, since $i$ is opportune.
\end{itemize}
\end{enumerate}
\end{proof}

\section{Proofs of the main results}\label{sec3}

\subsection{Proof of Theorem \ref{mainTheo}(1)}\label{subsec3P1}

It is easy to see that \cite[Theorem 2.6]{Heg05a} is equivalent to the following: \enquote{For a finite \textit{solvable} group $G$ with $\l_{-1}(G)\geq\rho$, we have that $\dl(G)\leq\max(2,\log_{3/4}(2\rho)+3)$.}. This implies more generally that $\dl(\Rad(G))\leq\max(2,\log_{3/4}(2\rho)+3)$ for a finite group $G$ with $\l_{-1}(G)\geq\rho$, since $\l_{-1}(\Rad(G))\geq\l_{-1}(G)$ as $\l_{-1}$ is CS-increasing (see Example \ref{groupTheoEx}(2) and Remark \ref{groupTheoRem}(1)).

It remains to show that for a finite group $G$ with $\l_{-1}(G)\geq\rho$, we have $[G:\Rad(G)]\leq\rho^{E_1}$. To this end, we will show that for finite semisimple groups $H$, $\L_{-1}(H)\leq|H|^{\t(E_0)}$. The assertion then    follows by an application of Lemma \ref{almostSolvLem} with $f:=\l_{-1}$.

The proof of the asserted upper bound on $\L_{-1}(H)$, in turn, is by an application of Lemma \ref{poBoundLem} with $F:=\L_{-1}$, so it suffices to prove that $\L_{-1}(T)\leq|T|^{E_0}$ for nonabelian finite characteristically simple groups $T$.

Write $T=S^n$ with $S$ a nonabelian finite simple group and $n\in\mathbb{N}^+$. We prove that $\L_{-1}(S^n)\leq|S^n|^{E_0}$ by induction on $n$, the induction base $n=1$ following from Corollary \ref{fulCor}. For the induction step, fix an automorphism $\alpha=(\alpha_1\times\cdots\times\alpha_n)\circ\sigma$ of $S^n$ such that $\L_{-1}(\alpha)=\L_{-1}(S^n)$. The decomposition of $\sigma$ into disjoint cycles corresponds to a direct decomposition of $S^n$ into $\alpha$-invariant normal subgroups, whence by the induction hypothesis, we may assume that $\sigma$ is an $n$-cycle, say w.l.o.g~the cycle $(1,2,\ldots,n)$.

Straightforward computations reveal that an element $(s_1,\ldots,s_n)\in S^n$ is inverted by $\alpha$ if and only if $\alpha_1(s_n)=s_1^{-1}$ and for $i=2,\ldots,n-1$, $\alpha_i(s_{i-1})=s_i^{-1}$. These relations imply that for each $s_1\in S$, there is at most one tuple in $S^n$ inverted by $\alpha$ and having first component $s_1$, whence $\L_{-1}(\alpha)\leq|S|=|S^n|^{1/n}\leq|S^n|^{1/2}<|S|^{E_0}$, as required.\qed

\subsection{Proof of Theorem \ref{mainTheo}(2)}\label{subsec3P2}

By Lemma \ref{almostSolvLem}, in order to prove that $\l_2(G)\geq\rho$ implies that $[G:\Rad(G)]\leq\rho^{-4}$, it suffices to show that $\L_2(H)\leq|H|^{3/4}$ for all finite semisimple groups $H$. To see that this holds, fix an automorphism $\alpha$ of $H$ such that $\L_2(\alpha)=\L_2(H)$. We show that $\L_2(\alpha)\leq|H|^{3/4}$ in a simple case distinction:

\begin{enumerate}
\item Case: $|\fix(\alpha)|\leq|H|^{1/4}$. Then by Lemma \ref{lELem} and \cite[Theorem 9]{GR06a}, we have $\L_2(\alpha)\leq\k(H)\cdot|\fix(\alpha)|\leq|H|^{1/2}\cdot|H|^{1/4}=|H|^{3/4}$.

\item Case: $|\fix(\alpha)|>|H|^{1/4}$. Then by Lemma \ref{lTwoLem} with $N:=G:=H$, we have $\L_2(\alpha)\leq[H:\fix(\alpha)]<|H|^{3/4}$.
\end{enumerate}

It remains to show that $\l_2(G)\geq\rho$ implies that $\dl(\Rad(G))\leq2\cdot\log_{3/4}(\rho)+1$. Although, as remarked in Example \ref{groupTheoEx}(3), the function $\l_2$ is not C-submultiplicative, this will be an application of Lemma \ref{dlBoundLem}. The \enquote{trick} is to replace $\l_2$ by an evaluation-wise larger function $\f$ which is C-submultiplicative and apply Lemma \ref{dlBoundLem} to $\f$.

Explicitly, $\f$ is defined as follows: For automorphisms $\alpha,\beta$ of a finite group $G$, we set $\Func(\alpha,\beta):=|\P_2(\alpha\mid \id,\beta)|$ (with notation as in Proposition \ref{shiftProp}), $\Func(G):=\max_{\alpha,\beta\in\Aut(G)}(\Func(\alpha,\beta))$ as well as $\f:=\Func_{\rel}$. Since $\Func(\alpha,\id)=\L_2(\alpha)$, we have that $\Func\geq\L_2$, and thus $\f\geq\l_2$. Hence if we can show that the condition $\f(G)\geq\rho$ implies the mentioned upper bound on $\dl(\Rad(G))$, the stronger condition $\l_2(G)\geq\rho$ implies it as well.

Observing that by Proposition \ref{shiftProp}, for $N\cha G$ and elements $g\in G$, $n\in N$, we have $g,ng\in \P_2(\alpha\mid \id,\beta)$ if and only if $g\in \P_2(\alpha\mid \id,\beta)$ and $n\in \P_2(\alpha_{\mid N}\mid \id_N,(\tau_g\circ\beta)_{\mid N})$, we see with an $N$-coset-wise counting argument that $\Func$, and hence $\f$, is C-submultiplicative. For a successful application of Lemma \ref{dlBoundLem}, it remains to show that for a finite solvable group $H$ with $\dl(H)\geq 2$, we have $\f(H)\leq 3/4$. We show that in general, a finite group $G$ such that $\f(G)>3/4$ is abelian.

To this end, fix automorphisms $\alpha,\beta$ of $G$ such that $|\P_2(\alpha\mid \id,\beta)|>3/4\cdot|G|$. Observe that if $g,h\in G$ are such that $g,h,gh\in \P_2(\alpha\mid \id,\beta)$, then $h$ and $\beta(g)$ commute. Hence for fixed $h\in \P_2(\alpha\mid \id,\beta)$, since $|\P_2(\alpha\mid \id,\beta)\cap \P_2(\alpha\mid \id,\beta)h^{-1}|>1/2\cdot|G|$, the centralizer of $h$ has cardinality greater than $1/2\cdot|G|$ and thus coincides with $G$ by Lagrange's theorem. It follows that $\P_2(\alpha\mid \id,\beta)\subseteq\zeta G$, whence by another application of Lagrange's theorem, $\zeta G=G$, as required.\qed

\subsection{Proof of Theorem \ref{mainTheo}(3)}\label{subsec3P3}

Fix $\rho\in\left(0,1\right]$. We would like to show that the orders of finite semisimple groups $H$ such that $\l_3(H)\geq\rho$ are bounded. Since $H$ embeds into $\Aut(\Soc(H))$, this problem is equivalent to bounding $|\Soc(H)|$. Write $\Soc(H)=S_1^{n_1}\times\cdots\times S_r^{n_r}$, where the $S_i$ are pairwise nonisomorphic nonabelian finite simple groups and $n_i\in\mathbb{N}^+$ for $i=1,\ldots,r$. The task of bounding $|\Soc(H)|$ can be split up into the following two subtasks:

\begin{enumerate}
\item Prove that $\max(|S_1|,\ldots,|S_r|)$ is bounded (\enquote{order bound}).

\item Prove that $\max(n_1,\ldots,n_r)$ is bounded (\enquote{exponent bound}).
\end{enumerate}

We will tackle these two tasks one after the other, but first, we make the following observation to ease notation: View $H$ as a subgroup of $\Aut(\Soc(H))=\Aut(S_1^{n_1})\times\cdots\times\Aut(S_r^{n_r})$. For $i=1,\ldots,r$, denote by $P_i\leq\Aut(\Soc(H))$ the product of the direct factors $\Aut(S_j^{n_j})$ of $\Aut(\Soc(H))$ for $j\not=i$. Set $C_i:=H\cap P_i$. Then $C_i$ is characteristic in $H$; set $H_i:=H/C_i$. By one of the isomorphism theorems, $H_i$ can be identified with a subgroup of $\Aut(\Soc(H))/P_i=\Aut(S_i^{n_i})$ containing $S_i^{n_i}$. Since $\l_3$ is CQ-increasing, we conclude that $\l_3(H_i)\geq\l_3(H)\geq\rho$, and so in carrying out the two subtasks above, we may assume w.l.o.g.~that $r=1$.

Assume thus henceforth that $H$ is a finite semisimple group with characteristically simple socle, say $\Soc(H)=S^n$ for a nonabelian finite simple group $S$ and $n\in\mathbb{N}^+$, and assume that $\l_3(H)\geq\rho$. By Theorem \ref{almostSimpleTheo}, we can fix a constant $C>0$ such that for all nonabelian finite simple groups $S$ with $|S|\geq C$, we have $\frac{\L_3(\Aut(S))}{|S|}\leq|S|^{-0.053}$. Denote by $O$ the maximum outer automorphism group order of a nonabelian finite simple group of order at most $\max(C,\rho^{-1/0.053})$. We will show that

\begin{equation}\label{orderBound}
|S|\leq\max(C,\rho^{-1/0.053})
\end{equation}

and that

\begin{equation}\label{exponentBound}
n\leq 16\cdot\frac{\log_{60}(\rho)}{-0.053}+\log_{(1-\frac{1}{\max(C,\rho^{-1/0.053})})}(\frac{\rho}{O^{16\cdot\frac{\log_{60}(\rho)}{-0.053}}}).
\end{equation}

Using henceforth the notation of Subsection \ref{subsec2P6} throughout, fix an automorphism $\alpha$ of $H$.

For establishing Equation (\ref{orderBound}), assume that $|S|>\max(C,\rho^{-1/0.053})$. Then $|S|>\rho^{-1/0.118}$, and so by Lemma \ref{opportuneLem}(3), cosets $K\beta$ of $K$ in $H$ (where, if $K\beta$ contains any elements cubed by $\alpha$, $\beta$ is chosen to be such an element) that are distinct from $K$ contain less than $\rho|K|$ many elements cubed by $\alpha$ (this is because for such cosets, $\sigma_{\beta}\not=\id$, whence there exists at least one opportune $i\in\{1,\ldots,n\}$). Hence if we can also show that the number of elements of $K$ that are cubed by $\alpha$ is less than $\rho|K|$, we have a contradiction. Note that if $\sigma_{\alpha}\not=\id$, then even under $\sigma_{\beta}=\id$, there still exist opportune indices $i$, and we are done. On the other hand, if $\sigma_{\alpha}=\id$, and thus $\alpha=\alpha_1\times\cdots\times\alpha_n$, then since $|S|>C$, the unique extension of $\alpha_{\mid K}$ to an automorphism of $\Aut(S)^n$ only cubes at most $|S|^{(1-0.053)\cdot n}$ elements in all of $\Aut(S)^n$, and so $\alpha$ only cubes at most a fraction of $|S|^{-0.053\cdot n}\leq|S|^{-0.053}<\rho$ of the elements of $K$, as we wanted to show.

We will also establish Equation (\ref{exponentBound}) by contradiction, so assume that $n$ is larger than the right-hand side of Equation (\ref{exponentBound}). Again, we will reach a contradiction by showing that for each coset $K\beta$ of $K$ in $H$, the number of elements of $K\beta$ that are cubed by $\alpha$ (i.e., the cardinality of the set $K_{\beta}$) is less than $\rho|K|$. We do so in a case distinction according to the value of the number $M$ of opportune indices $i\in\{1,\ldots,n\}$:

\begin{enumerate}
\item Case: $M>16\cdot\frac{\log_{60}(\rho)}{-0.053}$. Then $\lceil\frac{M}{16}\rceil>\frac{\log_{60}(\rho)}{-0.118}$, or equivalently, $60^{-0.118\lceil M/16\rceil}<\rho$, so in view of Lemma \ref{opportuneLem}(3), we are done.

\item Case: $M\leq 16\cdot\frac{\log_{60}(\rho)}{-0.053}$. Note that by assumption, we then have

\begin{equation}\label{nMIneq}
(1-\frac{1}{\max(C,\rho^{-1/0.053})})^{n-M}\cdot O^M<\rho.
\end{equation}

Assume w.l.o.g. that the set of common fixed points of $\sigma_{\alpha}$ and $\sigma_{\beta}$ is $\{1,\ldots,n-M\}$. Let $\pi:\Aut(S)^n\rightarrow\Aut(S)^{n-M}$ be the projection onto the first $n-M$ coordinates. Since we may of course assume that $K\beta$ contains at least one element cubed by $\alpha$, we have $\alpha(\beta)=\beta^3$ by the above convention on the choice of $\beta$. Hence by Equation (\ref{coordinateEq}), we have, for each $k=(k_1,\ldots,k_n)\in K_{\beta}$, that

\begin{equation}\label{coordinateEq2}
\alpha_i(k_i)=k_i\beta_i(k_i)\beta_i^2(k_i)
\end{equation}

for all $i=1,\ldots,n-M$. We use this to prove that 

\begin{equation}\label{piIneq}
|\pi[K_{\beta}]|\leq(1-\frac{1}{\max(C,\rho^{-1/0.053})})^{n-M}\cdot|\pi[K]|.
\end{equation}

To see that Equation (\ref{piIneq}) holds, we count $\pi[S^n]$-coset-wise in $\pi[K]$. Fix such a coset $\pi[S^n]\kappa$ with $\kappa=(\kappa_1,\ldots,\kappa_{n-M})$, assuming w.l.o.g. that $\kappa\in\pi[K_{\beta}]$ and thus by Equation (\ref{coordinateEq2}), $\alpha_i(\kappa_i)=\kappa_i\beta_i(\kappa_i)\beta_i^2(\kappa_i)$, i.e., $\kappa_i\in\P_3(\alpha_i\mid\id,\beta_i,\beta_i^{2})$ (notation from Proposition \ref{shiftProp}), for $i=1,\ldots,n-M$. Analogously, for $s=(s_1,\ldots,s_{n-M})\in S^{n-M}=\pi[S^n]$, we have $s\kappa\in\pi[K_{\beta}]$ only if $s_i\kappa_i\in\P_3(\alpha_i\mid\id,\beta_i,\beta_i^2)$ for all $i=1,\ldots,n-M$. Hence by Proposition \ref{shiftProp}, a necessary condition on $s\in S^{n-M}$ for $s\kappa\in\pi[K_{\beta}]$ to hold is that for $i=1,\ldots,n-M$,

\begin{equation}\label{coordinateEq3}
\alpha_i(s_i)=s_i\cdot(\tau_{\kappa_i}\circ\beta_i)(s_i)\cdot(\tau_{\kappa_i\beta_i(\kappa_i)}\circ\beta_i^2)(s_i).
\end{equation}

By Lemma \ref{nonShiftLem}, for each $i=1,\ldots,n-M$, Equation (\ref{coordinateEq3}) holds for at most $|S|-1$ values of $s_i\in S$, and so $\pi[S^n]$-coset-wise (and thus as a whole), the fraction of elements of $\pi[K]$ that lie in $\pi[K_{\beta}]$ is bounded from above by $(\frac{|S|-1}{|S|})^{n-M}=(1-\frac{1}{|S|})^{n-M}$. Using that $|S|\leq\max(C,\rho^{-1/0.053})$, Equation (\ref{piIneq}) follows.

Observing that $|S|^M\leq|\ker{\pi}|\leq|\Aut(S)|^M$ and using Equations (\ref{nMIneq}) and (\ref{piIneq}), we conclude that

\begin{align*}
\frac{|K_{\beta}|}{|K|} &\leq\frac{(1-\frac{1}{\max(C,\rho^{-1/0.053})})^{n-M}\cdot|\pi[K]|\cdot|\Aut(S)|^M}{|\pi[K]|\cdot|S|^M} \\ &\leq(1-\frac{1}{\max(C,\rho^{-1/0.053})})^{n-M}\cdot O^M<\rho,
\end{align*}

as we wanted to show.
\end{enumerate}\qed

\section{Concluding remarks}\label{sec4}

In retrospect, we proved several results that can be used for a general investigation of the functions $\L_e$ with $e\in\mathbb{Z}$ (such as Lemma \ref{lELem}, Proposition \ref{shiftProp} and Lemma \ref{almostSolvLem}). Still, the techniques used in the proofs of some of our results (such as Lemma \ref{lTwoLem} as well as Lemma \ref{lThreeLem} and Theorem \ref{almostSimpleTheo}, whose proof by Lemma \ref{lThreeLem} was basically reduced to the case $S=\A_1(q)$) were tailored for a particular exponent $e\in\{-1,2,3\}$, and we do not expect those results to have a straightforward generalization to other exponents.

The function $\L_{-1}$ allowed for a particularly nice treatment since it is C-submulti-\linebreak[4]plicative (whence Lemma \ref{poBoundLem} could be applied). However, we also saw in the proof of Theorem \ref{mainTheo}(2) (Subsection \ref{subsec3P2}) that a group-theoretic function $f$ need not be C-submultiplicative itself in order to make use of the respective results of Subsection \ref{subsec2P4} for the investigation of $f$; it suffices to find a C-submultiplicative group-theoretic function $f_0$ that \enquote{majorizes} $f$ (i.e., such that $f(G)\leq f_0(G)$ for all finite groups $G$) and to be able to apply the techniques for C-submultiplicative functions to $f_0$.

In view of this, it might be interesting to note that, as is easy to see with a coset-wise counting argument and using Proposition \ref{shiftProp}, for each $e\in\mathbb{N}^+$, the group-theoretic function $\Ll_e$, defined by $G\mapsto\max_{\alpha,\beta_1,\ldots,\beta_e\in\Aut(G)}(|\P_e(\alpha\mid\beta_1,\ldots,\beta_e)|)$ (for the notation, see Proposition \ref{shiftProp}), majorizes $\L_e$ and is C-submultiplicative. In particular, in order to show that under a condition of the form $\l_3(G)\geq\rho$ for some $\rho\in\left(0,1\right]$, $\dl(\Rad(G))$ is bounded, it would suffice to prove that for large enough numbers $\rho_0\in\left(0,1\right)$ and $k\in\mathbb{N}^+$, all finite solvable groups $H$ with $\dl(H)\geq k$ satisfy $\Ll_3(H)\leq\rho_0|H|$.

While hoping that these ideas will lead to extensions of our main results to other exponents $e\in\mathbb{Z}$, we do note that analoga of our main results do not exist for \textit{all} $e\in\mathbb{Z}$. For example, if there exists a nonabelian finite simple group $S$ such that $e \equiv 1 \Mod{\exp(S)}$ (the smallest such $e>1$ being $31$), then for all $n\in\mathbb{N}$, $\l_e(S^n)=\l_1(S^n)=1$, and so then not even demanding that $\l_e(G)=1$ is enough to ensure that $[G:\Rad(G)]$ is bounded.

\end{document}